\documentclass[11pt,letterpaper,reqno]{amsart}
\usepackage[OT2,T1]{fontenc}
\usepackage{amsmath, amsthm}
\usepackage{amsfonts}
\usepackage{amssymb}
\usepackage[all]{xy}
\usepackage{graphicx}
\usepackage{array}
\usepackage{url}
\usepackage{rotating}
\usepackage{fancyhdr}
\usepackage{xfrac}
\usepackage{xcolor}
\usepackage{cases}
\usepackage{multicol}
\usepackage{enumitem}

\title[Kato and Kuzumaki's properties for $\K$ of $p$-adic function fields]{On Kato and Kuzumaki's properties for the Milnor $\K$ of function fields of $p$-adic curves}
\author{Diego Izquierdo}
\address{CMLS, École polytechnique, F–91128 Palaiseau cedex, France}
\email{diego.izquierdo@polytechnique.edu}
\author{Giancarlo Lucchini Arteche}
\address{Departamento de Matem\'aticas, Facultad de Ciencias, Universidad de Chile, Las Palmeras 3425, \~Nu\~noa, Santiago, Chile}
\email{luco@uchile.cl}
 
\date{}

\DeclareSymbolFont{cyrletters}{OT2}{wncyr}{m}{n}
\DeclareMathSymbol{\Sha}{\mathalpha}{cyrletters}{"58}
\DeclareMathSymbol{\Brusse}{\mathalpha}{cyrletters}{"42}

\theoremstyle{plain}
\newtheorem{theorem}{Theorem}[section]
\newtheorem{MT}{Main Theorem}
\newtheorem{lemma}[theorem]{Lemma}
\newtheorem{proposition}[theorem]{Proposition}
\newtheorem{corollary}[theorem]{Corollary}

\theoremstyle{definition}
\newtheorem{remark}[theorem]{Remark}

\newtheorem{notation}[theorem]{Notation}

\newcommand{\cal}[1]{\mathcal{#1}}
\newcommand{\bb}[1]{\mathbb{#1}}
\newcommand{\Z}{\bb{Z}}
\newcommand{\Q}{\bb{Q}}
\newcommand{\K}{\mathrm{K}_2}
\newcommand{\im}{\mathrm{im}}
\newcommand{\res}{\mathrm{res}}
\newcommand{\ires}{i_{\mathrm{res}}}
\newcommand{\iram}{i_{\mathrm{ram}}}

\usepackage{marvosym}

\numberwithin{equation}{section}

\begin{document}

\begin{abstract}
Let $K$ be the function field of a curve $C$ over a $p$-adic field $k$. We prove that, for each $n, d \geq 1$ and for each hypersurface $Z$ in $\mathbb{P}^n_{K}$ of degree $d$ with $d^2 \leq n$, the second Milnor $K$-theory group of $K$ is spanned by the images of the norms coming from finite extensions $L$ of $K$ over which $Z$ has a rational point. When the curve $C$ has a point in the maximal unramified extension of $k$, we generalize this result to hypersurfaces $Z$ in $\mathbb{P}^n_{K}$ of degree $d$ with $d \leq n$.\\

\textbf{MSC Classes:} 11E76, 11G20, 12G05, 12G10, 14G05, 14G27, 14J45, 14J70, 19C99, 19F05.

\end{abstract}

\maketitle

\section{Introduction}

In 1986, in the article \cite{KK}, Kato and Kuzumaki stated a set of conjectures which aimed at giving a diophantine characterization of cohomological dimension of fields. For this purpose, they introduced some properties of fields which are variants of the classical $C_i$-property and which involve Milnor $K$-theory and projective hypersurfaces of small degree. They hoped that those properties would characterize fields of small cohomological dimension.\\

More precisely, fix a field $K$ and two non-negative integers $q$ and $i$. Let $\mathrm{K}_q(K)$ be the $q$-th Milnor $K$-group of $K$. For each finite extension $L$ of $K$, one can define a norm morphism $N_{L/K}: \mathrm{K}_q(L)\rightarrow \mathrm{K}_q(K)$ (see Section 1.7 of \cite{Kat}). Thus, if $Z$ is a scheme of finite type over $K$, one can introduce the subgroup $N_q(Z/K)$ of $\mathrm{K}_q(K)$ generated by the images of the norm morphisms $N_{L/K}$ when $L$ runs through the finite extensions of $K$ such that $Z(L) \neq \emptyset$. One then says that the field $K$ is $C_i^q$ if, for each $n \geq 1$, for each finite extension $L$ of $K$ and for each hypersurface $Z$ in $\mathbb{P}^n_{L}$ of degree $d$ with $d^i \leq n$, one has $N_q(Z/L) = \mathrm{K}_q(L)$. For example, the field $K$ is $C_i^0$ if, for each finite extension $L$ of $K$, every hypersurface $Z$ in $\mathbb{P}^n_{L}$ of degree $d$ with $d^i \leq n$ has a 0-cycle of degree 1. The field $K$ is $C_0^q$ if, for each tower of finite extensions $M/L/K$, the norm morphism $N_{M/L}: \mathrm{K}_q(M)\rightarrow \mathrm{K}_q(L)$ is surjective.\\

Kato and Kuzumaki conjectured that, for $i \geq 0$ and $q\geq 0$, a perfect field is $C_i^q$ if, and only if, it is of cohomological dimension at most $i+q$. This conjecture generalizes a question raised by Serre in \cite{GC} asking whether the cohomological dimension of a $C_i$-field is at most $i$. As it was already pointed out at the end of Kato and Kuzumaki's original paper \cite{KK}, Kato and Kuzumaki's conjecture for $i=0$ follows from the Bloch-Kato conjecture (which has been established by Rost and Voevodsky, cf.~\cite{Riou}): in other words, a perfect field is $C_0^q$ if, and only if, it is of cohomological dimension at most $q$. However, it turns out that the conjectures of Kato and Kuzumaki are wrong in general. For example, Merkurjev constructed in \cite{Mer} a field of characteristic 0 and of cohomological dimension 2 which does not satisfy property $C^0_2$. Similarly, Colliot-Th\'el\`{e}ne and Madore produced in \cite{CTM} a field of characteristic 0 and of cohomological dimension 1 which did not satisfy property $C^0_1$. These counter-examples were all constructed by a method using transfinite induction due to Merkurjev and Suslin. The conjecture of Kato and Kuzumaki is therefore still completely open for fields that usually appear in number theory or in algebraic geometry.\\

In 2015, in \cite{Wit}, Wittenberg proved that totally imaginary number fields and $p$-adic fields have the $C_1^1$ property. In 2018, in \cite{diego}, the first author also proved that, given a positive integer $n$, finite extensions of $\mathbb{C}(x_1,\ldots,x_n)$ and of $\mathbb{C}(x_1,\ldots,x_{n-1})((t))$ are $C_i^q$ for any $i,q \geq 0$ such that $i+q =n$. These are essentially the only known cases of Kato and Kuzumaki's conjectures. Note however that a variant of the $C_1^q$-property involving homogeneous spaces under connected linear groups is proved to characterize fields with cohomological dimension at most $q+1$ in \cite{ILA}.\\

In the present article, we are interested in Kato and Kuzumaki's conjectures for the function field $K$ of a smooth projective curve $C$ defined over a $p$-adic field $k$. The field $K$ has cohomological dimension $3$, and hence it is expected to satisfy the $C_i^q$-property for $i+q\geq 3$. As already mentioned, the Bloch-Kato conjecture implies this result when $q\geq 3$. The cases $q=0$ and $q=1$ seem out of reach with the current knowledge, since they likely imply the $C_2^0$-property for $p$-adic fields, which is a widely open question. In this article, we make progress in the case $q=2$.\\

Our first main result is the following:

\begin{MT} \label{MTA}
Function fields of $p$-adic curves satisfy the $C_2^2$-property.
\end{MT}

\noindent Of course, this implies that function fields of $p$-adic curves also satisfy the $C_i^2$-property for each $i \geq 2$. It therefore only remains to prove the $C_1^2$-property. In that direction, we prove the following main result:

\begin{MT} \label{MTB}
Let $K$ be the function field of a smooth projective curve $C$ defined over a $p$-adic field $k$. Assume that $C$ has a point in the maximal unramified extension of $k$. Then, for each $n, d \geq 1$ and for each hypersurface $Z$ in $\mathbb{P}^n_{K}$ of degree $d$ with $d \leq n$, we have  $\K(K)=N_2(Z/K)$.
\end{MT}

\noindent This theorem applies for instance when $K$ is the rational function field $k(t)$ or more generally the function field of a curve that has a rational point.\\

Since the proofs of these theorems are quite involved, we provide here below an outline with some details. Section \ref{sec prel} introduces all the notations and basic definitions we will need in the sequel. In Section \ref{sec C22}, we prove Theorem \ref{corint}, which
widely generalizes Main Theorem \ref{MTA}. Finally, in Section \ref{sec4}, we prove Theorem \ref{endth}, and
its Corollaries \ref{MTBbis} and \ref{MTBbisbis}, which widely generalize Main Theorem \ref{MTB}.

\subsection*{Ideas of proof for Main Theorem \ref{MTA}} Let $K$ be the function field of a smooth projective curve $C$ defined over a $p$-adic field $k$. Take two integers $n, d \geq 1$ such that $d^2\leq n$, a hypersurface $Z$ in $\mathbb{P}^n_{K}$ of degree $d$ and an element $x \in \K(K)$. We want to prove that $x\in N_2(Z/K)$. To do so, we roughly proceed in four steps, that are inspired from the proof of the $C_1^1$ property for number fields in \cite{diego} but that require to deal with several new difficulties:
\begin{enumerate}
    \item \emph{Solve the problem locally:} for each closed point $v$ of $C$, prove that $x\in N_2(Z_{K_v}/K_v)$. This provides $r_v$ finite extensions $M_i^{(v)}/K_v$ such that $Z(M_i^{(v)})\neq \emptyset$ and
    $$x \in\left\langle N_{M_i^{(v)}/K_v}(\K(M_i^{(v)})) \mid 1 \leq i \leq r_v\right\rangle. $$
    \item \emph{Globalize the extensions $M_i^{(v)}/K_v$:} for each closed point $v$ of $C$ and each $1\leq i\leq r_v$, find a finite extension $K_i^{(v)}$ of $K$ contained in $M_i^{(v)}$ such that $Z(K_i^{(v)})\neq \emptyset$. Then prove that there exists a finite subset of these global extensions, say $K_1,\ldots,K_r$, such that for every closed point $v$ of $C$, $x$ lies in the subgroup of $\K(K_v)$ generated by the norms coming from the $(K_i\otimes_K K_v)$'s.
    \item \emph{Establish a local-to-global principle for norm groups:} prove the vanishing of the Tate-Shafarevich group:
    \begin{multline*}
\Sha_2 := \ker \left( \frac{\K(K)}{\left\langle N_{K_i/K}(\K(K_i)) \mid 1 \leq i \leq r \right\rangle} \right. \\ \left. \rightarrow  \prod_{v\in C^{(1)}} \frac{\K(K_v)}{\left\langle N_{K_i\otimes_K K_v/K_v}(\K(K_i\otimes_K K_v)) \mid 1 \leq i \leq r \right\rangle}  \right).
\end{multline*}
\item \emph{Conclude:} By step (2), we have $x\in \Sha_2$. Hence, step (3) implies that:
$$x \in \langle N_{K_i/K}(\K(K_i)) \mid 1 \leq i \leq r \rangle \subset N_2(Z/K), $$
as wished.
\end{enumerate}
Let us now briefly discuss the proofs of Steps (1), (2) and (3). Step (1) can be proved by combining some results for $p$-adic fields due to Wittenberg (\cite{Wit}) and the computation of the groups $\K(K_v)$ thanks to the residue maps in Milnor $K$-theory, cf.~\S\ref{subsubsec local}.

In the way it is written above, Step (2) can be easily deduced from Greenberg's approximation Theorem. However, as we will see below, we will need a stronger version of that step, that will require a completely different proof.

Step (3) is the hardest part of the proof. The first key tool that we use is a Poitou-Tate duality for motivic cohomology over the field $K$ proved by the first author in \cite{diego0}. This provides a finitely generated free Galois module $\hat{T}$ over $K$ such that the Pontryagyn dual of $\Sha_2$ is the quotient of:
$$\Sha^2(K,\hat{T}):=\ker \left( H^2(K,\hat{T})\rightarrow \prod_{v\in C^{(1)}} H^2(K_v,\hat{T}) \right)$$
by its maximal divisible subgroup. Now, a result of Demarche and Wei (\cite{DW}) states that, under some technical linear disjointness assumption for the extensions $K_i^{(v)}/K$, one can find two finite extensions $K'$ and $K''$ of $K$ such that the restriction:
$$\Sha^2(K,\hat{T}) \rightarrow \Sha^2(K',\hat{T})\oplus \Sha^2(K'',\hat{T})$$
is injective and $\hat{T}$ is a permutation Galois module over both $K'$ and $K''$. If the groups $\Sha^2(K',\hat{T})$ and $\Sha^2(K'',\hat{T})$ were trivial, then we would be done. But that is not the case in our context because the $p$-adic function field $K$ has finite extensions $K'$ such that $\Sha^2(K',\mathbb{Z})$ is not trivial (see for instance the appendix of \cite{ctps}). This ``failure of Chebotarev's Density Theorem'' makes the computation of $\Sha^2(K,\hat{T})$ very complicated and technical. By carrying out quite subtle Galois cohomology computations and by using some results of Kato (\cite{Kat}), we prove that, under some technical assumptions on the $K_i^{(v)}$ (see below) and another technical assumption on $C$ (which is trivially satisfied when $C(k)\neq\emptyset$), the group $\Sha^2(K,\hat{T})$ is always divisible, even though it might not be trivial, cf.~\S\ref{subsubsec calcul Sha}. This is enough to apply the Poitou-Tate duality and deduce the vanishing of $\Sha_2$.

Now, in order to ensure that the $K_i^{(v)}$'s and $C$ fulfill the conditions required to carry out the previous argument, we have to:
\begin{itemize}
    \item add a step (0) in which we reduce to the case where $C$ satisfies a technical assumption close to having a rational point; and
    \item modify the constructions of the $K_i^{(v)}$'s in Step (2), which cannot be done anymore by using Greenberg's approximation Theorem.
\end{itemize}
The reduction to the case where $C$ satisfies the required conditions uses the Beilinson-Lichtenbaum conjecture for motivic cohomology and a local-to-global principle due to Kato (cf.~\cite{Kat}) with respect to the places of $K$ that come from a suitable regular model of the curve $C$, cf.~\S\ref{subsubsec resind1}. As for Step (2), we want to construct the $K_i^{(v)}$'s so that they fulfill two extra conditions:
\begin{itemize}
    \item[(a)] one of the $K_i^{(v)}$'s has to be of the form $k_i^{(v)}K$ for some finite unramified extension $k_i^{(v)}/k$. This is achieved by observing that  $Z(k^{\mathrm{nr}}(C))\neq \emptyset$ since the field $k^{\mathrm{nr}}(C)$ is $C_2$ and $Z$ is a hypersurface in $\mathbb{P}^n_K$ of degree $d$ with $d^2 \leq n$.
     \item[(b)] the $K_i^{(v)}$'s have to satisfy some suitable linear disjointness conditions also involving abelian extensions of $K$ that are locally trivial everywhere. This is achieved by an approximation argument that uses the implicit function Theorem for $Z$ over the $K_v$'s, weak approximation and an analogue of Hilbert's irreducibility Theorem for the field $K$, cf.~\S\ref{subsubsec global}.
\end{itemize}
Note that, since we use the implicit function Theorem, the previous argument only works when the hypersurface $Z$ is smooth. We thus need to add an extra step to the proof in which we reduce to that case. This uses a dévissage technique that is due to Wittenberg (\cite{Wit}) and that requires to work with \emph{all} proper varieties over $K$ (instead of only hypersurfaces), cf.~\S\ref{subsubsec conclusion}. For that reason, we need to prove a wide generalization of Main Theorem \ref{MTA} to all proper varieties. This is the object of Theorem \ref{corint} in the core of the text. Of course, this requires to modify and generalize the proofs of Steps (1), (2) and (3) so that they can be applied in that more general setting.

\subsection*{Ideas of proof for Main Theorem \ref{MTB}}

The proof of Main Theorem \ref{MTB} follows by combining Main Theorem \ref{MTA} together with a result roughly stating that every element of $\K(K)$ can be written as a product of norms coming from extensions of the form $k'K$ with $k'$ a finite extension of $k$ whose ramification degree is fixed, cf. Theorem \ref{extell}. The general ideas to prove this last result are similar to (and a bit simpler than) those used in Main Theorem \ref{MTA}.

\subsection*{Acknowledgements.} We thank Jean-Louis Colliot-Thélène, Olivier Wittenberg and an anonymous referee for their comments and suggestions.

The second author's research was partially supported by ANID via FONDECYT Grant 1210010.

\section{Notations and preliminaries}\label{sec prel}

In this section we fix the notations that will be used throughout this article.

\subsection*{Milnor $K$-theory} 
Let $K$ be any field and let $q$ be a non-negative integer. The $q$-th Milnor $\mathrm{K}$-group of $K$ is by definition the group $\mathrm{K}_0(K)=\mathbb{Z}$ if $q =0$ and:
$$\mathrm{K}_q(K):= \underbrace{K^{\times} \otimes_{\mathbb{Z}} \ldots \otimes_{\mathbb{Z}} K^{\times}}_{q \text{ times}} / \left\langle x_1 \otimes \ldots \otimes x_q | \exists i,j, i\neq j, x_i+x_j=1 \right\rangle$$
if $q>0$. For $x_1,\ldots,x_q \in K^{\times}$, the symbol $\{x_1,\ldots,x_q\}$ denotes the class of $x_1 \otimes \cdots \otimes x_q$ in $\mathrm{K}_q(K)$. More generally, for $r$ and $s$ non-negative integers such that $r+s=q$, there is a natural pairing:
$$\mathrm{K}_r(K) \times \mathrm{K}_s(K) \rightarrow \mathrm{K}_q(K)$$
which we will denote $\{\cdot, \cdot\}$.\\

When $L$ is a finite extension of $K$, one can construct a norm homomorphism
\[N_{L/K}: \mathrm{K}_q(L) \rightarrow \mathrm{K}_q(K),\]
satisfying the following properties (see Section 1.7 of \cite{Kat} or Section 7.3 of \cite{GS}):
\begin{itemize}
\item[$\bullet$] For $q=0$, the map $N_{L/K}: \mathrm{K}_0(L) \rightarrow \mathrm{K}_0(K)$ is given by multiplication by $[L:K]$.
\item[$\bullet$] For $q=1$, the map $N_{L/K}: \mathrm{K}_1(L) \rightarrow \mathrm{K}_1(K)$ coincides with the usual norm $L^{\times} \rightarrow K^{\times}$.
\item[$\bullet$] If $r$ and $s$ are non-negative integers such that $r+s=q$, we have $N_{L/K}(\{x,y\})=\{x,N_{L/K}(y)\}$ for $x \in \mathrm{K}_r(K)$ and $y\in \mathrm{K}_s(L)$.
\item[$\bullet$] If $M$ is a finite extension of $L$, we have $N_{M/K} = N_{L/K} \circ N_{M/L}$.
\end{itemize}
Recall also that Milnor $K$-theory is endowed with residue maps (see Section 7.1 of \cite{GS}). Indeed, when $K$ is a henselian discrete valuation field with ring of integers $R$, maximal ideal $\mathfrak{m}$ and residue field $\kappa$, there exists a unique residue morphism:
$$\partial: \mathrm{K}_q(K) \rightarrow \mathrm{K}_{q-1}(\kappa)$$
such that, for each uniformizer $\pi$ and for all units $u_2,\ldots,u_q\in R^{\times}$ whose images in $\kappa$ are denoted $\overline{u_2},\ldots,\overline{u_q}$, one has:  $$\partial (\{\pi,u_2,\ldots,u_q\})=\{ \overline{u_2},\ldots,\overline{u_q}\}.$$
The kernel of $\partial$ is the subgroup $U_q(K)$ of $\mathrm{K}_q(K)$ generated by symbols of the form $\{x_1,\ldots,x_q\}$ with $x_1,\ldots,x_q \in R^\times$. If $U_q^1(K)$ stands for the subgroup of $\mathrm{K}_q(K)$ generated by those symbols that lie in $U_q(K)$ and that are of the form $\{x_1,\ldots,x_q\}$ with $x_1 \in 1+\mathfrak{m}$ and $x_2,\ldots,x_q \in K^{\times}$, then $U_q^1(K)$ is $\ell$-divisible for each prime $\ell$ different from the characteristic of $\kappa$ and  $U_q(K)/U_q^1(K)$ is canonically isomorphic to $\mathrm{K}_q(\kappa)$. Moreover, if $L/K$ is a finite extension with ramification degree $e$ and residue field $\lambda$, then the norm map $N_{L/K}:\mathrm{K}_q(L)\rightarrow \mathrm{K}_q(K)$ sends $U_q(L)$ to $U_q(K)$ and $U_q^1(L)$ to $U_q^1(K)$, and the following diagrams commute:
\begin{equation}\label{eqn diagr norm K-th}
\xymatrix{
    \mathrm{K}_q(L)/U_q(L) \ar[r]_-{\cong}^-{\partial_L} \ar[d]^{N_{L/K}} & \mathrm{K}_{q-1}(\lambda)\ar[d]^{N_{\lambda/\kappa}} &U_q(L)/U^1_q(L) \ar[r]_-{\cong} \ar[d]^{N_{L/K}} & \mathrm{K}_{q}(\lambda)\ar[d]^{eN_{\lambda/\kappa}}\\
    \mathrm{K}_q(K)/U_q(K) \ar[r]_-{\cong}^-{\partial_K} & \mathrm{K}_{q-1}(\kappa), & U_q(K)/U^1_q(K) \ar[r]_-{\cong} & \mathrm{K}_{q}(\kappa).}
\end{equation}

\subsection*{The $C_i^q$ properties}
Let $K$ be a field and let $i$ and $q$ be two non-negative integers. For each $K$-scheme $Z$ of finite type, we denote by $N_q(Z/K)$ the subgroup of $\mathrm{K}_q(K)$ generated by the images of the maps $N_{L/K}: \mathrm{K}_q(L) \rightarrow \mathrm{K}_q(K)$ when $L$ runs through the finite extensions of $K$ such that $Z(L)\neq \emptyset$. The field $K$ is said to have the $C_i^q$ property if, for each $n \geq 1$, for each finite extension $L$ of $K$ and for each hypersurface $Z$ in $\mathbb{P}^n_{L}$ of degree $d$ with $d^i \leq n$, one has $N_q(Z/L) = \mathrm{K}_q(L)$.

\subsection*{Motivic complexes}
Let $K$ be a field. For $i \geq 0$, we denote by $z^i(K, \cdot)$ Bloch’s cycle complex defined in \cite{Bloch}. The étale motivic complex $\mathbb{Z}(i)$ over $K$ is then defined as the complex of Galois modules $z^i(-,\cdot)[-2i]$. By the Nesterenko-Suslin-Totaro Theorem and the Beilinson-Lichtenbaum Conjecture, it is known that :
\begin{equation}\label{nesterenko}
H^i(K,\mathbb{Z}(i)) \cong \mathrm{K}_i(K),
\end{equation}
and
\begin{equation}\label{beilinsonlicht}
H^{i+1}(K,\mathbb{Z}(i)) =0,
\end{equation}
for all $i\geq 0$. Statement \eqref{nesterenko} was originally proved in \cite{NS} and \cite{Tot}, and statement \eqref{beilinsonlicht} was deduced from the Bloch-Kato conjecture in \cite{SV}, \cite{GL2} and \cite{GL}. The Bloch-Kato conjecture itself was proved in \cite{SJ} and \cite{Voe}. For the convenience of the reader, we also provide more tractable references: statement \eqref{nesterenko} follows from Theorem 5.1 of \cite{Haewei} and Theorem 1.2(2) of \cite{Gei}, and statement \eqref{beilinsonlicht} can be deduced from the Bloch-Kato conjecture as explained in Lemma 1.6 and Theorem 1.7 of \cite{Haewei}.

\subsection*{Fields of interest}
From now on and until the end of the article, $p$ stands for a prime number and $k$ for a $p$-adic field with ring of integers $\cal O_k$. We let $C$ be a smooth projective geometrically integral curve over $k$, and we let $K$ be its function field. We denote by $C^{(1)}$ the set of closed points in $C$. The residual index $\ires(C)$ of $C$ is defined to be the g.c.d. of the residual degrees of the $k(v)/k$ with $v\in C^{(1)}$. The ramification index $\iram(C)$ of $C$ is defined to be the g.c.d. of the ramification degrees of the $k(v)/k$ with $v\in C^{(1)}$. 

\subsection*{Tate-Shafarevich groups}
When $M$ is a complex of Galois modules over $K$ and $i \geq 0$ is an integer, we define the $i$-th Tate-Shafarevich group of $M$ as:
$$\Sha^i(K,M):=\ker \left( H^i(K,M) \rightarrow \prod_{v\in C^{(1)}} H^i(K_v,M) \right).$$
When a suitable regular model $\mathcal{C}/\cal O_k$ of $C/k$ is given, we also introduce the following smaller Tate-Shafarevich groups:
$$\Sha^i_{\mathcal{C}}(K,M):=\ker \left( H^i(K,M) \rightarrow \prod_{v\in \mathcal{C}^{(1)}} H^i(K_v,M) \right),$$
where $\mathcal{C}^{(1)}$ is the set of codimension $1$ points of $\mathcal{C}$.

\subsection*{Poitou-Tate duality for motivic cohomology}
We recall the Poitou-Tate duality for motivic complexes over the field $K$ (Theorem 0.1 of \cite{diego0} in the case $d=1$). Let $\hat{T}$ be a finitely generated free Galois module over $K$. Set $\check{T}:=\mathrm{Hom}(\hat{T},\mathbb{Z})$ and $T=\check{T}\otimes \mathbb{Z}(2)$. Then there is a perfect pairing of finite groups:
\begin{equation}\label{PT}
    \overline{\Sha^2(K,\hat{T})} \times \Sha^3(K,T) \rightarrow \mathbb{Q}/\mathbb{Z},
\end{equation}
where $\overline{A}$ denotes the quotient of $A$ by its maximal divisible subgroup.

Note that, in the case $\hat{T}=\mathbb{Z}$, the Beilinson-Lichtenbaum conjecture \eqref{beilinsonlicht} implies the vanishing of $\Sha^3(K,\mathbb{Z}(2))$ and hence the group $\Sha^2(K, \mathbb{Z})$ is divisible. By Shapiro's Lemma, the same holds for the group $\Sha^2(K, \mathbb{Z}[E/K])$ for every étale $K$-algebra $E$.

\section{On the $C_2^2$-property for $p$-adic function fields}\label{sec C22}

The goal of this section is to prove the following theorem:

\begin{theorem}\label{corint}
Let $l/k$ be a finite unramified extension and set $L:=lK$. Let $Z$ be a proper $K$-variety. Then the quotient:
$$\K(K)/\langle N_{L/K}(\K(L)), N_2(Z/K) \rangle$$
is $\chi_K(Z,E)^2$-torsion for each coherent sheaf $E$ on $Z$.
\end{theorem}

Here, $\chi_K(Z,E)$ denotes the Euler characteristic of $E$ over $Z$. Main Theorem \ref{MTA} can be deduced as a very particular case of Theorem \ref{corint}, in which this characteristic is trivial. We explain this at the end of the section.

\subsection{Proof of Theorem \ref{corint}}

\subsubsection{Step 0: Interpreting norms in Milnor $K$-theory in terms of motivic cohomology}\label{subsubsec trad Sha}

The following lemma, which will be extensively used in the sequel, allows to interpret quotients of $\K(K)$ by norm subgroups as twisted motivic cohomology groups.

\begin{lemma}\label{abstractlemma}
Let $L$ be a field and let $L_1,\ldots,L_r$ be finite separable extensions of $L$. Consider the étale $L$-algebra $E:=\prod_{i=1}^r L_i$ and let $\check{T}$ be the Galois module defined by the following exact sequence:
\begin{equation}\label{seqabs}
    0 \rightarrow \check{T} \rightarrow \mathbb{Z}[E/L] \rightarrow \mathbb{Z} \rightarrow 0.
\end{equation}
Then:
$$H^3(L,\check{T}\otimes \mathbb{Z}(2)) \cong \K(L)/\langle N_{L_i/L}(\K(L_i)) \; | \; 1\leq i \leq r \rangle .$$
\end{lemma}

\begin{proof}
Exact sequence \eqref{seqabs} induces a distinguished triangle:
$$\check{T}\otimes \mathbb{Z}(2) \rightarrow \mathbb{Z}[E/L]\otimes \mathbb{Z}(2) \rightarrow \mathbb{Z}(2) \rightarrow \check{T}\otimes \mathbb{Z}(2)[1].$$
By taking cohomology, we get an exact sequence:
$$H^2(L,\mathbb{Z}[E/L]\otimes \mathbb{Z}(2))) \rightarrow H^2(L,\mathbb{Z}(2)) \rightarrow H^3(L,\check{T}\otimes \mathbb{Z}(2)) \rightarrow H^3(L,\mathbb{Z}[E/L]\otimes \mathbb{Z}(2)).$$
By Shapiro's Lemma, we have: 
\begin{gather*}
    H^2(L,\mathbb{Z}[E/L]\otimes \mathbb{Z}(2)))\cong H^2(E,\mathbb{Z}(2)),\\ H^3(L,\mathbb{Z}[E/L]\otimes \mathbb{Z}(2)))\cong H^3(E,\mathbb{Z}(2)).
\end{gather*} 
Moreover, as recalled in section \ref{sec prel}, the Nesterenko-Suslin-Totaro Theorem and the Beilinson-Lichtenbaum conjecture give the following isomorphisms:
\begin{gather*}
    H^2(L,\mathbb{Z}(2))\cong \K(L),\\ 
    H^2(E,\mathbb{Z}(2)) \cong \prod_{i=1}^r \K(L_i),\\
    H^3(E,\mathbb{Z}(2))=0.
\end{gather*} 
 We therefore get an exact sequence:
$$\prod_{i=1}^r \K(L_i) \rightarrow \K(L) \rightarrow H^3(L,\check{T}\otimes \mathbb{Z}(2)) \rightarrow 0,$$
in which the first map is the product of the norms. 
\end{proof}

\subsubsection{Step 1: Reducing to curves with residual index $1$}\label{subsubsec resind1}

In this step, we prove the following proposition, that allows to reduce to the case when the curve $C$ has residual index $1$:

\begin{proposition}\label{resid}
Let $k'/k$  be the unramified extension of $k$ of degree $\ires(C)$ and set $K':=k'K$. Then the norm morphism $N_{K'/K}: \K(K')\rightarrow \K(K)$ is surjective.
\end{proposition}

\begin{proof}
Consider the Galois module $\check{T}$ defined by the following exact sequence:
$$0 \rightarrow \check{T}\rightarrow \mathbb{Z}[K'/K] \rightarrow \mathbb{Z}  \rightarrow 0,$$
Since $K'/K$ is cyclic, a $\Z$-basis of $\check T$ is given by $s^\alpha-s^{\alpha-1}$ with $s$ a generator of $\mathrm{Gal}(K'/K)$ and $1\leq \alpha\leq \ires(C)-1$. Then the arrow $\mathbb{Z}[K'/K] \rightarrow  \check{T}$ that sends $s$ to $s-1$ gives rise to an exact sequence:
$$0 \rightarrow  \mathbb{Z}\rightarrow \mathbb{Z}[K'/K] \rightarrow  \check{T} \rightarrow 0,$$
and hence to a distinguished triangle:
$$ \mathbb{Z}(2) \rightarrow \mathbb{Z}[K'/K]\otimes \mathbb{Z}(2) \rightarrow  \check{T}\otimes\mathbb{Z}(2)  \rightarrow \mathbb{Z}(2)[1].$$
By the Beilinson-Lichtenbaum conjecture, the group $H^3(K',\mathbb{Z}(2))$ is trivial. Hence we get an inclusion:
$$\Sha^3_{\mathcal{C}}(K,\check{T}\otimes \mathbb{Z}(2))\subseteq \Sha^4_{\mathcal{C}}(K,\mathbb{Z}(2)),$$
where $\mathcal{C}$ is a fixed regular, proper and flat model of $C$ whose reduced special fiber $C_0$ is a strict normal crossing divisor.
Now, the distinguished triangle
\[\Z(2)\to \Q(2)\to \Q/\Z(2)\to \Z(2)[1],\]
and the vanishing of the groups $H^3(K,\Q(2))$ and $H^4(K,\Q(2))=0$ (which follow from Lemma 2.5 and Theorem 2.6.c of \cite{kahn}) give rise to an isomorphism
\[\Sha^3_{\mathcal{C}}(K,\Q/\Z(2))\cong \Sha^4_{\mathcal{C}}(K,\mathbb{Z}(2)),\]
and by Proposition 5.2 of \cite{Kato}, the group on the left is trivial, and hence so is the former group.

Now observe that, by Lemma \ref{abstractlemma}, we have:
$$\Sha^3_{\mathcal{C}}(K,\check{T}\otimes \mathbb{Z}(2)) \cong \ker \left( \K(K)/\mathrm{im}(N_{K'/K}) \rightarrow \prod_{v\in \mathcal{C}^{(1)}} \K(K_v)/\mathrm{im}(N_{K'_v/K_v}) \right).$$ 
We claim that the extension $K'/K$ totally splits at each place $v\in \mathcal{C}^{(1)}$. From this, we deduce that:
$$0=\Sha^3_{\mathcal{C}}(K,\check{T}\otimes \mathbb{Z}(2))\cong \K(K)/\mathrm{im}(N_{K'/K}),$$
and hence the norm morphism $N_{K'/K}: \K(K')\rightarrow \K(K)$ is surjective.

It remains to check the claim. It is obviously satisfied for $v\in C^{(1)}$, so we may and do assume $v \in \mathcal{C}^{(1)} \setminus C^{(1)}$. If $\kappa$ and $\kappa'$ denote the residue fields of $k$ and $k'$, we then have to prove that all the irreducible components of $C_0$ are $\kappa'$-curves. To do so, consider an infinite sequence of finite unramified field extensions $k=k_0 \subset k_1 \subset k_2 \subset \ldots $ all with degrees prime to $[k':k]$ and denote by $\kappa=\kappa_0 \subset \kappa_1 \subset \kappa_2 \subset \ldots$ the corresponding residue fields. Let $k_{\infty}$ (resp. $\kappa_{\infty}$) be the union of all the $k_i$'s (resp. $\kappa_i$'s). Since $\kappa_\infty$ is infinite, Lemma 4.6 of \cite{Wit} and the definition of $\ires(C)$ imply that each irreducible component of $C_0 \times_{\kappa_0} \kappa_\infty$ has index divisible by $[k':k]$. Hence the same is true for all the irreducible components of $C_0$. But recall that, by the Lang-Weil estimates, any smooth geometrically integral variety defined over a finite field has a zero-cycle of degree $1$. We deduce that the irreducible components of $C_0$ are $\kappa'$-curves. 
\end{proof}

\subsubsection{Step 2: Solving the problem locally}\label{subsubsec local}

In this step, we prove that the analogous statement to Theorem \ref{corint} over the completions of $K$ holds. For that purpose, we first need to settle a simple lemma:

\begin{lemma}\label{lemloc}
Let $l/k$ be a finite extension and set $K_0:=k((t))$ and $L_0:=l((t))$. The residue map $\partial: \K(K_0) \rightarrow k^\times$ induces an isomorphism:
$$\K(K_0)/N_{L_0/K_0}(\K(L_0)) \cong k^\times/N_{l/k}(l^\times).$$
\end{lemma}

\begin{proof}
We have the following commutative diagram from \eqref{eqn diagr norm K-th}:
\begin{equation*}
    \xymatrix{
    \K(L_0)\ar[r]^-{\partial_{L_0}}\ar[d]_{N_{L_0/K_0}} & l^\times \ar[d]^{N_{l/k}}\\
    \K(K_0) \ar[r]^-{\partial_{K_0}} & k^\times.
    }
\end{equation*}
Recalling that $U_2(K_0)$ is by definition the kernel of $\partial_{K_0}$ (cf.~\S\ref{sec prel}), this diagram induces an exact sequence:
$$
0 \rightarrow \frac{U_2(K_0)}{U_2(K_0) \cap N_{L_0/K_0}(\K(L_0))} \rightarrow \frac{\K(K_0)}{N_{L_0/K_0}(\K(L_0))} \xrightarrow{\bar \partial_{K_0}} \frac{k^\times}{N_{l/k}(l^\times)} \rightarrow 0.
$$
It therefore suffices to prove that $U_2(K_0)=U_2(K_0) \cap N_{L_0/K_0}(\K(L_0))$. For that purpose, recall that we have a commutative diagram with exact lines:
\begin{equation*}
    \xymatrix{
    0 \ar[r] & U_2^1(L_0) \ar[d]^{N_{L_0/K_0}} \ar[r]& U_2(L_0) \ar[d]^{N_{L_0/K_0}}\ar[r]& \K(l) \ar[d]^{N_{l/k}}\ar[r] & 0\\
    0 \ar[r] & U_2^1(K_0)  \ar[r]& U_2(K_0) \ar[r]& \K(k) \ar[r] & 0.
    }
\end{equation*}
But the map $N_{l/k}: \K(l)\rightarrow \K(k)$ is surjective since $p$-adic fields have the $C_0^2$-property, and the map $N_{L_0/K_0}: U_2^1(L_0)\rightarrow U_2^1(K_0)$ is surjective since the group $U_2^1(K_0)$ is divisible (as explained in \S\ref{sec prel}). We deduce that $N_{L_0/K_0}: U_2(L_0)\rightarrow U_2(K_0)$ is also surjective, as wished.
\end{proof}

\begin{proposition}\label{proploc}
Let $l/k$ be a finite unramified extension and set $K_0:=k((t))$ and $L_0:=l((t))$. Let $Z$ be a proper $K_0$-variety. Then the quotient:
$$\K(K_0)/\langle N_{L_0/K_0}(\K(L_0)), N_2(Z/K_0) \rangle$$
is $\chi_{K_0}(Z,E)$-torsion for each coherent sheaf $E$ on $Z$.
\end{proposition}

\begin{proof}
For each proper ${K_0}$-scheme $Z$, we denote by $n_Z$ the exponent of the quotient group $\K({K_0})/\langle N_{L_0/K_0}(\K(L_0)), N_2(Z/K_0) \rangle$. We say that $Z$ satisfies property $(P)$ if it has a model over $\mathcal{O}_{K_0}$ that is irreducible, regular, proper and flat. To prove the proposition, it suffices to check assumptions (1), (2) and (3) of Proposition 2.1 of \cite{Wit}. \\

Assumption (1) is obvious. Assumption (3) is a direct consequence of Gabber and de Jong's Theorem (Theorem 3 of the Introduction of \cite{ILO}). It remains to check assumption (2). For that purpose, we proceed in the same way as in the proof of Theorem 4.2 of \cite{Wit}. Indeed, consider a proper $K_0$-scheme $X$ together with a model $\mathcal{X}$ that is irreducible, regular, proper and flat and denote by $Y$ its special fiber. Let $m$ be the multiplicity of $Y$ and let $D$ be the effective divisor on $\mathcal{X}$ such that $Y=mD$. \\

The residue map induces an exact sequence:
\begin{equation}\label{exs}
0 \rightarrow \frac{U_2(K_0)}{U_2(K_0) \cap N_2(X/K_0)} \rightarrow \frac{\K(K_0)}{N_2(X/K_0)} \rightarrow \frac{\mathrm{K}_1(k)}{\partial (N_2(X/K_0))} \rightarrow 0.
\end{equation}
Moreover:
\begin{itemize}
    \item[(a)] since $k$ satisfies the $C_0^2$ property, the proof of Lemma 4.4 of \cite{Wit} still holds in our context, and hence the group $\frac{U_2(K_0)}{U_2(K_0) \cap N_2(X/K_0)}$ is killed by the multiplicity $m$ of the special fiber $Y$ of $\mathcal{X}$;
    \item[(b)] the proof of Lemma 4.5 of \cite{Wit} also holds in our context, and hence $\partial (N_2(X/K_0))=N_1(Y/k)=N_1(D/k)$;
    \item[(c)] by Corollary 5.4 of \cite{Wit} applied to the proper $k$-scheme $D \sqcup \mathrm{Spec}(l)$, the group $k^{\times}/\langle N_{l/k}(l^{\times}), N_1(D/k) \rangle$ is killed by $\chi_k(D,\mathcal{O}_D)$.
\end{itemize}
By using exact sequence \eqref{exs}, facts (b) and (c) and Lemma \ref{lemloc}, we deduce that:
$$\chi_k(D,\mathcal{O}_D) \cdot \K(K_0) \subset \langle N_{L_0/K_0}(\K(L_0)), N_2(X/K_0), U_2(K_0) \rangle.$$
Hence, by fact (a), we get:
$$m\chi_k(D,\mathcal{O}_D) \cdot \K(K_0) \subset \langle N_{L_0/K_0}(\K(L_0)), N_2(X/K_0) \rangle.$$
But $m\chi_k(D,\mathcal{O}_D)=\chi_{K_0}(X,\mathcal{O}_X)$ by Proposition 2.4 of \cite{ELW}, and hence the quotient $\K(K_0)/\langle N_{L_0/K_0}(\K(L_0)), N_2(X/K_0) \rangle$ is killed by $\chi_{K_0}(X,\mathcal{O}_X)$.
\end{proof}

\subsubsection{Step 3: Globalizing local field extensions}\label{subsubsec global}

In rest of the proof, we will show how one can deduce the global Theorem \ref{corint} from the local Proposition \ref{proploc}. For that purpose, we first need to find a suitable way to globalize local extension: more precisely, given a place $w \in C^{(1)}$ and a finite extension $M^{(w)}$ of $K_w$ such that $Z(M^{(w)}) \neq \emptyset$, we want to find a suitable finite extension $M$ of $K$ that can be seen as a subfield of $M^{(w)}$ and such that $Z(M)\neq \emptyset$. For technical reasons related to the failure of Cébotarev's Theorem over the field $K$, we also need $M$ to be linearly disjoint from a given finite extension of $K$. The following proposition is the key statement allowing to do that:

\begin{proposition}\label{locglob}
Let $Z$ be a smooth geometrically integral $K$-variety. Let $T$ be a finite subset of $C^{(1)}$. Fix a finite extension $L$ of $K$ and, for each $w\in T$, a finite extension $M^{(w)}$ of $K_w$ such that $Z(M^{(w)})\neq \emptyset$. Then there exists a finite extension $M$ of $K$ satisfying the following properties:
\begin{itemize}
    \item[(i)] $Z(M)\neq\emptyset$;
    \item[(ii)] for each $w\in T$, there exists a $K$-embedding $M\hookrightarrow M^{(w)}$;
    \item[(iii)] the extensions $L/K$ and $M/K$ are linearly disjoint.
\end{itemize}
\end{proposition}

\begin{proof}
Before starting the proof, we introduce the following notations for each $w \in T$:
\begin{gather*}
    n^{(w)}:=[M^{(w)}:K_w],\\
    m^{(w)}:=\prod_{w' \in T \setminus \{w\}} n^{(w')},
\end{gather*}
so that the integer $n:=n^{(w)}m^{(w)}$ is independent of $w$. We now proceed in three substeps.\\

\emph{Substep 1.} By Proposition 4.9 in Chapter I of \cite{Hart}, there exists a projective hypersurface $Z'$ in $\mathbb{P}^m_K$ given by a non-zero equation $$f(x_0,\ldots,x_m)=0$$ that is birationally equivalent to $Z$. Let $U$ and $U'$ be non-empty open sub-schemes of $Z$ and $Z'$ that are isomorphic. Up to reordering the variables and shrinking $U'$, we may and do assume that the polynomial $\partial f /\partial x_0$ is non-zero and that:
$$U' \cap \{\partial f/\partial x_0(x_0,\ldots,x_m)=0\} = \emptyset.$$

Given an element $w\in T$, the variety $Z$ is smooth, $Z(M^{(w)})\neq \emptyset$ and $M^{(w)}$ is large (for the definition of this notion, please refer to \cite{Pop}). Hence the sets $U(M^{(w)})$ and $U'(M^{(w)})$ are non-empty. We can therefore find a non-trivial solution $(y_0^{(w)},\ldots,y_m^{(w)})$ of the equation $f(x_0,\ldots,x_m)=0$ in $M^{(w)}$ such that: 
$$\begin{cases}
(y_0^{(w)},\ldots,y_m^{(w)})\in U'\\
\partial f/\partial x_0(y_0^{(w)},\ldots,y_m^{(w)})\neq 0.
\end{cases}$$

\emph{Substep 2.} Given $w \in T$, there exist $m^{(w)}$ elements $\alpha_1,\ldots,\alpha_{m^{(w)}} \in M^{(w)}$ whose respective minimal polynomials $\mu_{\alpha_1},\ldots,\mu_{\alpha_{m^{(w)}}}$ are pairwise distinct and such that $M^{(w)}=K_w(\alpha_i)$ for each $1
\leq i\leq m^{(w)}$. Recalling that $n=n^{(w)}m^{(w)}$, introduce the degree $n$ monic polynomial $\mu^{(w)}:=\prod_{i=1}^{m^{(w)}} \mu_{\alpha_i}$ and consider the set $H$ of $n$-tuples $(a_0,\ldots,a_{n-1})\in K^n$ such that the polynomial $T^{n}+\sum_{i=0}^{n-1}a_iT^i$ is irreducible over $L$. By Corollary 12.2.3 of \cite{FA}, the set $H$ contains a Hilbertian subset of $K^{n}$, and hence, according to Proposition 19.7 of \cite{Jarden}, if we fix some $\epsilon>0$, we can find an $n$-tuple $(b_0,\ldots,b_{n-1})$ in $H$ such that the polynomial $\mu:=T^{n}+\sum_{i=0}^{n-1}b_iT^i$ is coefficient-wise $\epsilon$-close to $\mu^{(w)}$ for each $w \in T$. Consider the field $K':=K[T]/(\mu)$. If $\epsilon$ is chosen small enough, then there exists a $K$-embedding $K'\hookrightarrow M^{(w)}$ for each $w\in T$ by Krasner's Lemma (cf.~Lemma 8.1.6 in \cite{NSW}). Moreover, since $(b_0,\ldots,b_{n-1})\in H$, the polynomial $\mu$ is irreducible over $L$, and hence the extensions $K'/K$ and $L/K$ are linearly disjoint. \\

\emph{Substep 3.} According to Substep 1, for each $w\in T$, $y_0^{(w)}$ is a simple root of the polynomial \[g^{(w)}(T):=f(T,y_1^{(w)},\ldots,y_m^{(w)}).\]
Let $H'$ be the set of $m$-tuples $(z_1,\ldots,z_m)$ in $K'$ such that $f(T,z_1,\ldots,z_m)$ is irreducible over $LK'$. By Corollary 12.2.3 of \cite{FA}, the set $H'$ contains a Hilbertian subset of $K'^{m}$. Hence, by Proposition 19.7 of \cite{Jarden}, we can find $(y_1,\ldots,y_m)$ in $H'$ such that the polynomial \[g(T):=f(T,y_1,\ldots,y_m)\] is coefficient-wise $\epsilon$-close to $g^{(w)}$ for each $w\in T$. Introduce the field $M:=K'[T]/(g(T))$. We check that $M$ satisfies the conditions of the proposition, provided that $\epsilon$ is chosen small enough:
\begin{itemize}
    \item[(i)] Fix $w\in T$. By Substep 1, the $m$-tuple $(y_0^{(w)},\ldots,y_m^{(w)})$ lies in $U'$. Hence, for $\epsilon$ small enough, if $y_{0,w}$ stands for the root of $g$ that is closest to $y_0^{(w)}$, then the $m$-tuple $(y_{0,w},y_1,\ldots,y_m)$ lies in $U'$. We deduce that $U'(M)\neq \emptyset$, and hence $Z(M)\neq \emptyset$.
    \item[(ii)] For each $w\in T$, the polynomial $g^{(w)}$ has a simple root in $M^{(w)}$, and hence so does $g(T)$ if $\epsilon$ is chosen small enough, again by Krasner's Lemma. The field $M$ can therefore be seen as a subfield of $M^{(w)}$.
    \item[(iii)] Since $(y_1,\ldots,y_m)\in H'$, the polynomial $g(T)$ is irreducible over $LK'$. Hence the extensions $M/K'$ and $LK'/K'$ are linearly disjoint. Moreover, by Substep 2, the extensions $K'/K$ and $L/K$ are linearly disjoint. We deduce that $L/K$ and $M/K$ are linearly disjoint.
\end{itemize}
\end{proof}

\subsubsection{Step 4: Computation of a Tate-Shafarevich group}\label{subsubsec calcul Sha}

This step, which is quite technical, consists in computing the Tate-Shafarevich groups of some finitely generated free Galois modules over $K$. Recall that for each abelian group $A$, we denote by $\overline{A}$ the quotient of $A$ by its maximal divisible subgroup.

\begin{proposition}\label{shadiv}
Let $r\geq 2$ be an integer and let $L,K_1,\ldots,K_r$ be finite extensions of $K$ contained in $\overline{K}$. Consider the composite fields $K_{\mathcal{I}}:=K_1\cdots K_r$ and $K_{\hat{i}}:=K_1\cdots K_{i-1}K_{i+1}\cdots K_r$ for each $i$, and denote by $n$ the degree of $L/K$. Consider the Galois module $\hat{T}$ defined by the following exact sequence:
\begin{equation}\label{hat}
0 \rightarrow  \mathbb{Z}\rightarrow \mathbb{Z}[E/K] \rightarrow  \hat{T} \rightarrow 0,
  \end{equation}
where $E:=L\times K_1 \times \cdots \times K_r$. Given two positive integers $m$ and $m'$, make the following assumptions:
\begin{enumerate}[label=\upshape(LD\arabic*),ref=\upshape(LD\arabic*)]
\item\label{ld1} the Galois closure of $L/K$ and the extension $K_{\mathcal{I}}/K$ are linearly disjoint;
\item\label{h2} for each $i \in \{1,\ldots,r\}$, the fields $K_{i}$ and $K_{\hat{i}}$ are linearly disjoint over $K$;
\end{enumerate}

\begin{enumerate}[label=\upshape(H\arabic*),ref=\upshape(H\arabic*)]
\item\label{h1} the restriction map:
$$\Sha^2(K,\hat{T}) \rightarrow \Sha^2(L,\hat{T})\oplus \Sha^2(K_{\mathcal{I}},\hat{T})$$
is injective;

\item\label{h3} the restriction map $$\mathrm{Res}_{LK_{\mathcal{I}}/K_{\mathcal{I}}}: \Sha^2(K_{\mathcal{I}},\mathbb{Z}) \rightarrow \Sha^2(LK_{\mathcal{I}},\mathbb{Z})$$  is surjective and its kernel is $m$-torsion;

\item\label{h4} for each $i$, the restriction maps $$\mathrm{Res}_{LK_i/K_i}: \Sha^2(K_i,\mathbb{Z}) \rightarrow \Sha^2(LK_i,\mathbb{Z})$$
and
$$\mathrm{Res}_{LK_{\hat{i}}/K_{\hat{i}}}: \Sha^2(K_{\hat{i}},\mathbb{Z}) \rightarrow \Sha^2(LK_{\hat{i}},\mathbb{Z})$$
are surjective;

\item\label{h5} for each finite extension $L'$ of $L$ contained in the Galois closure of $L/K$, the kernel of the restriction map $$\mathrm{Res}_{L'K_{\mathcal{I}}/L}: \Sha^2(L',\mathbb{Z}) \rightarrow \Sha^2(L'K_{\mathcal{I}},\mathbb{Z})$$ is $m'$-torsion.
\end{enumerate} 
Then $\overline{\Sha^2(K,\hat{T})}$ is $\left((m\vee m')\wedge n\right)$-torsion.
\end{proposition}

Recall that $\overline{A}$ denotes the quotient of $A$ by its maximal divisible subgroup.

\begin{remark}
In the sequel of the article, we will only use the proposition in the case when $L/K$ is Galois. However, this assumption does not simplify the proof.
\end{remark}

\begin{proof}
Consider the following sequence:
\begin{equation}\label{keyseq}
\xymatrix @R=.3pc{
\Sha^2(K,\hat{T}) \ar[r]^-{f_0}& \Sha^2(L,\hat{T})\oplus \Sha^2(K_{\mathcal{I}},\hat{T}) \ar[r]^-{g_0}& \Sha^2(LK_{\mathcal{I}},\hat{T})\\
x \ar@{|->}[r]& (\mathrm{Res}_{L/K}(x),\mathrm{Res}_{K_{\mathcal{I}}/K}(x))&\\
&(x,y) \ar@{|->}[r] & \mathrm{Res}_{LK_{\mathcal{I}}/L}(x) - \mathrm{Res}_{LK_{\mathcal{I}}/K_{\mathcal{I}}}(y).
}
\end{equation}
It is obviously a complex, and the first arrow is injective by \ref{h1}. In order to give further information about the complex \eqref{keyseq}, let us consider the following commutative diagram, in which the first and second rows are obtained in the same way as the third:
\begin{equation}\label{diagram}
\xymatrix{
& 0\ar[d] & 0\ar[d]\\
 \Sha^2(K,\mathbb{Z}) \ar[r]^-{f_1}\ar[d]^-{\phi_0}& \Sha^2(L,\mathbb{Z})\oplus \Sha^2(K_{\mathcal{I}},\mathbb{Z}) \ar@{->>}[r]^-{g_1}\ar[d]^-{\phi_1}& \Sha^2(LK_{\mathcal{I}},\mathbb{Z})\ar[d]^-{\phi_2}\\
 \Sha^2(K, \mathbb{Z}[E/K]) \ar[r]^<<<<<{f} \ar[d]^-{\psi_0}& \Sha^2(L, \mathbb{Z}[E/K])\oplus \Sha^2(K_{\mathcal{I}}, \mathbb{Z}[E/K]) \ar[r]^<<<<<{g}\ar[d]^-{\psi_1}& \Sha^2(LK_{\mathcal{I}}, \mathbb{Z}[E/K])\ar[d]^-{\psi_2}\\
 \Sha^2(K,\hat{T}) \ar@{^{(}->}[r]^>>>>>>>>>>>>>>>>{f_0}& \Sha^2(L,\hat{T})\oplus \Sha^2(K_{\mathcal{I}},\hat{T}) \ar[r]^>>>>>>>>>>>{g_0}\ar[d]& \Sha^2(LK_{\mathcal{I}},\hat{T})\ar[d]\\
 & 0 & 0
}
\end{equation}

The second and third columns are exact since the exact sequence (\ref{hat}) splits over $L$, $K_{\mathcal{I}}$ and $LK_{\mathcal{I}}$.  Moreover, all the lines are complexes, and in the first one, the arrow $g_1$ is surjective since the restriction map:
$$ \Sha^2(K_{\mathcal{I}},\mathbb{Z}) \rightarrow \Sha^2(LK_{\mathcal{I}},\mathbb{Z})$$
is surjective by assumption \ref{h3}. 

The next two lemmas constitute the core of the proof of Proposition \ref{shadiv}.

\begin{lemma}\label{provient}
Let $a\in \Sha^2(K,\hat{T}) $ and $b=(b_L,b_{K_{\mathcal{I}}}) \in \Sha^2(L, \mathbb{Z}[E/K])\oplus \Sha^2(K_{\mathcal{I}}, \mathbb{Z}[E/K]))$ such that $f_0(a)=\psi_1(b)$ and $g(b)=0$. Then $mb_{K_{\mathcal{I}}}$ comes by restriction from $ \Sha^2(K_{\hat{i}}, \mathbb{Z}[E/K])$ for each $i$.
\end{lemma}

\begin{proof}
Consider the following commutative diagram, constructed exactly in the same way as diagram \eqref{diagram}:
\[\xymatrix{
& 0\ar[d] & 0\ar[d]\\
 \Sha^2(K,\mathbb{Z}) \ar[r]^-{f_1^i}\ar[d]^-{\phi_0}& \Sha^2(L,\mathbb{Z})\oplus \Sha^2(K_{\hat{i}},\mathbb{Z}) \ar@{->>}[r]^-{g_1^i}\ar[d]^-{\phi_1^i}& \Sha^2(LK_{\hat{i}},\mathbb{Z})\ar[d]^-{\phi_2^i}\\
 \Sha^2(K, \mathbb{Z}[E/K]) \ar[r]^<<<<<{f^i} \ar[d]^-{\psi_0}& \Sha^2(L, \mathbb{Z}[E/K])\oplus \Sha^2(K_{\hat{i}}, \mathbb{Z}[E/K]) \ar[r]^<<<<<{g^i}\ar[d]^-{\psi_1^i}& \Sha^2(LK_{\hat{i}}, \mathbb{Z}[E/K])\ar[d]^-{\psi_2^i}\\
 \Sha^2(K,\hat{T}) \ar@{^{(}->}[r]^>>>>>>>>>>>>>>>>{f_0^i}& \Sha^2(L,\hat{T})\oplus \Sha^2(K_{\hat{i}},\hat{T}) \ar[r]^>>>>>>>>>>>{g_0^i}\ar[d]& \Sha^2(LK_{\hat{i}},\hat{T})\ar[d]\\
 & 0 & 0
}\]
The last two columns are exact since the exact sequence (\ref{hat}) splits over $L$, $K_{\hat{i}}$ and $LK_{\hat{i}}$, and the restriction morphism $\Sha^2(K_{\hat{i}},\mathbb{Z}) \rightarrow \Sha^2(LK_{\hat{i}},\mathbb{Z})$ is surjective by \ref{h4}. Hence there exists $b_{K_{\hat{i}}} \in \Sha^2(K_{\hat{i}}, \mathbb{Z}[E/K])$ such that $\psi_1^i(b_L,b_{K_{\hat i}})=f_0^i(a)$ and $g^i(b_L,b_{K_{\hat i}})=0$. The pair: $$(0,b_{K_{\mathcal{I}}}-\mathrm{Res}_{K_{\mathcal{I}}/K_{\hat{i}}}(b_{K_{\hat{i}}})) \in \Sha^2(L, \mathbb{Z}[E/K])\oplus \Sha^2(K_{\mathcal{I}}, \mathbb{Z}[E/K])$$ then lies in $\ker(g) \cap \ker(\psi_1)$ and a diagram chase in \eqref{diagram} shows that there exists $c\in \Sha^2(K_{\mathcal{I}},  \mathbb{Z})$ such that:
$$\begin{cases}
\phi_1(0,c)=(0,b_{K_{\mathcal{I}}}-\mathrm{Res}_{K_{\mathcal{I}}/K_{\hat{i}}}(b_{K_{\hat{i}}}))\\
\mathrm{Res}_{LK_{\mathcal{I}}/K_{\mathcal{I}}}(c)=0.
\end{cases}$$
By \ref{h3}, we have $mc=0$, and hence:
$m\cdot (b_{K_{\mathcal{I}}}-\mathrm{Res}_{K_{\mathcal{I}}/K_{\hat i}}(b_{K_{\hat i}}))=0$.
\end{proof}

\begin{lemma}\label{verytech}
Set $\mu:=m\vee m'$ and take $a\in \Sha^2(K,\hat{T})$. Then $\mu a \in \mathrm{Im}(\psi_0)$.
\end{lemma}

Before proving the lemma, let us introduce some notation.

\begin{notation}
\begin{itemize}
    \item[(i)] For each $i$, we can find a family $(K_{ij})_j$ of finite extensions of $K_\mathcal{I}$ together with embeddings $\sigma_{ij}: K_i \hookrightarrow K_{ij}$ so that $K_{i,1}=K_\mathcal{I}$, the embedding $\sigma_{i,1}$ is the natural embedding $K_i\hookrightarrow K_\mathcal{I}$, and the $K$-algebra homomorphism:
    \begin{align*}    
K_i \otimes_K K_\mathcal{I} &\rightarrow \prod_j K_{ij}\\
x \otimes y &\mapsto (\sigma_{ij}(x)y)_j
\end{align*} 
is an isomorphism. We denote by $\tilde{\sigma}_{ij}: K_\mathcal{I} \rightarrow K_{ij}$ the embedding obtained by tensoring $\sigma_{ij}$ with the identity of $K_{\hat{i}}$. This is well-defined by \ref{h2}.
\item[(ii)] For each $i,j$, we can find a family $(L_{ijj'})_{j'}$ of finite extensions of $K_{ij}$ together with embeddings $\sigma_{ijj'}: L \hookrightarrow L_{ijj'}$ so that the $K$-algebra homomorphism:
    \begin{align} \label{tenseur}
    \begin{split}
L \otimes_K K_{ij} &\rightarrow \prod_{j'} L_{ijj'}\\
x \otimes y &\mapsto (\sigma_{ijj'}(x)y)_{j'},
\end{split}
\end{align}
is an isomorphism. We denote by $\tilde{\sigma}_{ijj'}: LK_i \rightarrow L_{ijj'}$ the embedding obtained by tensoring $\sigma_{ijj'}$ with $\sigma_{ij}$.
Observe that, when $j=1$, the $K$-algebra homomorphism \eqref{tenseur} is simply the isomorphism $L \otimes_K K_{\mathcal{I}} \cong LK_\mathcal{I}$, so that $\sigma_{i,1,1}$ is none other than the inclusion of $L$ in $LK_\mathcal{I}$.
    \item[(iii)] We can find a family of finite extensions $(L_{\alpha})_\alpha$ of $L$ together with embeddings $\tau_{\alpha}: L \hookrightarrow L_\alpha$ so that $L_1=L$, the embedding $\tau_1$ is the identity of $L$, and the $K$-algebra homormorphism:
    \begin{align*}    
L \otimes_K L &\rightarrow \prod_\alpha L_\alpha\\
x \otimes y &\mapsto (\tau_{\alpha}(x)y)_\alpha
\end{align*} 
is an isomorphism. For each $\alpha$, we denote by $\tilde{\tau}_\alpha: LK_\mathcal{I} \rightarrow L_\alpha K_\mathcal{I}$ the embedding obtained by tensoring $\tau_\alpha$ with the identity of $K_\mathcal{I}$. This is well-defined by \ref{ld1}.
\end{itemize}
\end{notation}

\begin{proof}
By Shapiro's Lemma, one can identify the second line of diagram \eqref{diagram} with the following complex:
\[\xymatrix{\displaystyle
\Sha^2(L,\mathbb{Z}) \oplus \bigoplus_i \Sha^2(K_i,\mathbb{Z}) \ar[d]^f \\
\displaystyle\bigoplus_\alpha \Sha^2(L_\alpha,\mathbb{Z}) \oplus \bigoplus_i \Sha^2(LK_i,\mathbb{Z}) \oplus \Sha^2(LK_{\mathcal{I}},\mathbb{Z}) \oplus \bigoplus_{i,j} \Sha^2(K_{ij},\mathbb{Z}) \ar[d]^g \\
\displaystyle\bigoplus_{\alpha}\Sha^2(L_{\alpha}K_{\mathcal I},\mathbb{Z}) \oplus \bigoplus_{i,j,j'} \Sha^2(L_{ijj'},\mathbb{Z})
}\]
where $f$ is given by:
$$(x,(y_i)_i) \mapsto \left(\left((\mathrm{Res}_{\tau_\alpha:L \hookrightarrow L_\alpha}(x)\right)_\alpha, \left(\mathrm{Res}_{LK_i/K_i}(y_i)\right)_i,  \mathrm{Res}_{LK_{\mathcal{I}}/L}(x), \left(\mathrm{Res}_{\sigma_{ij}: K_i \hookrightarrow K_{ij}}(y_i)\right)_{i,j}\right),$$
and $g$:
\begin{align*}
((&x_{\alpha})_{\alpha}, (y_i)_i, z, (t_{ij})_{i,j}) \mapsto\\& \left( \left(\mathrm{Res}_{L_{\alpha}K_\mathcal{I}/L_{\alpha}}(x_\alpha)-\mathrm{Res}_{ \tilde{\tau}_\alpha:LK_\mathcal{I}\hookrightarrow L_{\alpha}K_\mathcal{I}}(z)\right)_{\alpha},  \left( \mathrm{Res}_{\tilde{\sigma}_{ijj'}:LK_i\hookrightarrow L_{ijj'}}(y_i) - \mathrm{Res}_{L_{ijj'}/K_{ij}}(t_{ij}) \right)_{i,j} \right).
\end{align*}
Now take:
$$((x_{\alpha})_{\alpha}, (y_i)_i, z, (t_{ij})_{i,j})\in \bigoplus_\alpha \Sha^2(L_\alpha,\mathbb{Z}) \oplus \bigoplus_i \Sha^2(LK_i,\mathbb{Z}) \oplus \Sha^2(LK_{\mathcal{I}},\mathbb{Z}) \oplus \bigoplus_{i,j} \Sha^2(K_{ij},\mathbb{Z})$$
such that:
\begin{equation}\label{eqn hip1 lema}
\psi_1((x_{\alpha})_{\alpha}, (y_i)_i, z, (t_{ij})_{i,j})=f_0(a).
\end{equation}
Since $g_0(f_0(a))=0$ and $g_1$ is surjective, a diagram chase in \eqref{diagram} allows to assume that:
\begin{equation}\label{eqn hip2 lema}
((x_{\alpha})_{\alpha}, (y_i)_i, z, (t_{ij})_{i,j})\in \ker(g).
\end{equation}
This implies that:
\begin{numcases}{}
\mathrm{Res}_{L_{\alpha}K_\mathcal{I}/L_{\alpha}}(x_\alpha)=\mathrm{Res}_{ \tilde{\tau}_\alpha:LK_\mathcal{I}\hookrightarrow L_{\alpha}K_\mathcal{I}}(z),\quad\forall\, \alpha,\label{eqz}\\
  \mathrm{Res}_{\tilde{\sigma}_{ijj'}:LK_i\hookrightarrow L_{ijj'}}(y_i) = \mathrm{Res}_{L_{ijj'}/K_{ij}}(t_{ij}),\quad\forall\,i,j,j'. \label{eqt}
\end{numcases}
In particular,
\begin{equation}\label{eqn eqx1}
\mathrm{Res}_{L_{1}K_\mathcal{I}/L_{1}}(x_1)=\mathrm{Res}_{LK_\mathcal{I}/L}(x_1)=z,
\end{equation}
and hence the commutativity of the following diagram of field extensions:
\begin{equation*}
    \xymatrix@=1em{
    & L_\alpha K_\mathcal{I}\ar@{-}[dr]^{\tilde{\tau}_\alpha} \ar@{-}[dl] & \\
    L_\alpha\ar@{-}[dr]_{\tau_\alpha} & & LK_\mathcal{I}\ar@{-}[dl] \\
    & L &
    }
\end{equation*}
shows that:
\begin{align*}
    \mathrm{Res}_{L_{\alpha}K_\mathcal{I}/L_{\alpha}}(\mathrm{Res}_{\tau_\alpha:L \hookrightarrow L_{\alpha}}(x_1))&=\mathrm{Res}_{\tilde{\tau}_\alpha:LK_\mathcal{I}\hookrightarrow L_{\alpha}K_\mathcal{I}}(\mathrm{Res}_{LK_\mathcal{I}/L}(x_1))\\&=\mathrm{Res}_{\tilde{\tau}_\alpha:LK_\mathcal{I}\hookrightarrow L_{\alpha}K_\mathcal{I}}(z)\\&=\mathrm{Res}_{L_{\alpha}K_\mathcal{I}/L_{\alpha}}(x_\alpha).
\end{align*}
Since the kernel of $\mathrm{Res}_{L_{\alpha}K_\mathcal{I}/L_{\alpha}}$ is $m'$-torsion by \ref{h5}, we have:
\begin{equation}\label{eqn eqxalpha}
m'\mathrm{Res}_{\tau_\alpha:L \hookrightarrow L_{\alpha}}(x_1)=m'x_\alpha
\end{equation}
for all $\alpha$. Moreover, by \ref{h4}, one can find for each $i$ an element $\tilde{y}_i\in\Sha^2(K_i,\Z)$ such that: 
\begin{equation}\label{eqn eqy}
y_i=\mathrm{Res}_{LK_i/K_i}(\tilde{y}_i).
\end{equation}
Let us check that:
\begin{equation}\label{0}
\mu(( x_\alpha)_\alpha, ( y_i)_i,  z, ( t_{ij})_{i,j}) = \mu f\left( x_1,( \tilde{y}_i)_i\right).
\end{equation}

By construction (cf.~equations \eqref{eqn eqxalpha}, \eqref{eqn eqy} and \eqref{eqn eqx1}), we have:
\begin{gather*}
\mu (\mathrm{Res}_{\tau_\alpha:L \hookrightarrow L_{\alpha}}(x_1))_{\alpha}=\mu( x_\alpha)_\alpha,\\
(y_i)_i=\left(\mathrm{Res}_{LK_i/K_i}(\tilde{y}_i)\right)_i,\\
\mu\mathrm{Res}_{LK_{\mathcal{I}}/L}( x_1)=\mu z.
\end{gather*}
To finish the proof of \eqref{0}, it is therefore enough to check that:
\begin{equation}\label{4}
m t_{ij}=m \mathrm{Res}_{\sigma_{ij}: K_i \hookrightarrow K_{ij}}(\tilde{y}_i)
\end{equation}
for each $i$ and $j$. For that purpose, fix $i=i_0$, and consider first the case $j=1$. We then have $K_{i_0,1}=K_{\mathcal{I}}$, and hence, by using \eqref{eqt}:
\begin{multline*}
\mathrm{Res}_{LK_{\mathcal{I}}/K_{\mathcal{I}}}(t_{i_0,1})=\mathrm{Res}_{L_{i_0,1,1}/LK_{i_0}}(y_{i_0})\\ = \mathrm{Res}_{L_{i_0,1,1}/K_{i_0}}(\tilde{y}_{i_0}) = \mathrm{Res}_{LK_{\mathcal{I}}/K_{\mathcal{I}}}(\mathrm{Res}_{K_{i_0,1}/K_{i_0}}(\tilde{y}_{i_0})).
\end{multline*}
By \ref{h3}, we deduce that:
$$mt_{i_0,1}=m\mathrm{Res}_{K_{i_0,1}/K_{i_0}}(\tilde{y}_{i_0})=m\mathrm{Res}_{K_{\mathcal{I}}/K_{i_0}}(\tilde{y}_{i_0}).$$
Now move on to case of arbitrary $j$. By Lemma \ref{provient} together with equations \eqref{eqn hip1 lema} and \eqref{eqn hip2 lema}, the element
\[m(z,(t_{i,j})_{i,j})\in \Sha^2(LK_{\mathcal{I}},\mathbb{Z}) \oplus \bigoplus_{i,j} \Sha^2(K_{ij},\mathbb{Z}) = \Sha^2(K_\mathcal{I},\mathbb{Z}[E/K])\]
comes by restriction from $\Sha^2(K_{\hat{i}},\mathbb{Z}[E/K])$ for each $i$. In particular, the element:
\[(mt_{i_0,j})_j\in \bigoplus_j \Sha^2(K_{i_0,j},\mathbb{Z}) = \Sha^2(K_\mathcal{I},\mathbb{Z}[K_{i_0}/K])\]
comes by restriction from an element $t_{i_0}\in \Sha^2(K_{\mathcal{I}},\mathbb{Z})=\Sha^2(K_{\hat{i}_0},\mathbb{Z}[K_{i_0}/K])$. In other words:
$$(mt_{i_0,j})_j = (\mathrm{Res}_{\tilde{\sigma}_{i_0,j}: K_{\mathcal{I}} \hookrightarrow K_{i_0,j}}(t_{i_0}))_j. $$
In particular, $mt_{i_0,1} = t_{i_0}$, and hence for each $j$:
\begin{align*}
   mt_{i_0,j} &=  \mathrm{Res}_{\tilde{\sigma}_{i_0,j}: K_\mathcal{I} \hookrightarrow K_{i_0,j}}(t_{i_0})\\
   &=\mathrm{Res}_{\tilde{\sigma}_{i_0,j}: K_\mathcal{I} \hookrightarrow K_{i_0,j}}(mt_{i_0,1})\\
   &=\mathrm{Res}_{\tilde{\sigma}_{i_0,j}: K_\mathcal{I} \hookrightarrow K_{i_0,j}}(m\mathrm{Res}_{K_\mathcal{I}/K_{i_0}}(\tilde{y}_{i_0}))\\
      &=m\mathrm{Res}_{\sigma_{i_0,j}: K_{i_0} \hookrightarrow K_{i_0,j}}(\tilde{y}_{i_0}).
\end{align*}
This finishes the proofs of equalities \eqref{4} and \eqref{0}. Applying $\psi_1$ to \eqref{0} we deduce that:
$$\mu f_0(\alpha)=\mu f_0(\psi_0(( x_\alpha)_\alpha, ( y_i)_i,  z, ( t_{ij})_{i,j})).$$
Since $f_0$ is injective, we get:
$$\mu \alpha=\mu \psi_0(( x_\alpha)_\alpha, ( y_i)_i,  z, ( t_{ij})_{i,j}),$$
which finishes the proof of the lemma.
\end{proof}

We can now finish the proof of Proposition \ref{shadiv}. As recalled at the end of section \ref{sec prel}, the group $\Sha^2(K, \mathbb{Z}[E/K])$ is divisible and hence, by Lemma \ref{verytech}:
$$(m\vee m')\cdot  \Sha^2(K,\hat{T}) \subseteq \Sha^2(K,\hat{T})_{\mathrm{div}}.$$
In other words, the group $\overline{\Sha^2(K,\hat{T})}$ is $(m\vee m')$-torsion.

On the other hand, using once again the end of section \ref{sec prel}, the group $\overline{\Sha^2(L,\hat{T})}$ vanishes. Hence, by restriction-corestriction, $\overline{\Sha^2(K,\hat{T})}$ is $n$-torsion. We deduce that $\overline{\Sha^2(K,\hat{T})}$ is $\left((m\vee m')\wedge n\right)$-torsion.
\end{proof}

The following lemma will often allow us to check assumptions \ref{h3} and \ref{h4} of Proposition \ref{shadiv}:

\begin{lemma}\label{res1}
Let $l$ be a finite unramified extension of $k$ of degree $n$ and set $L=lK$. The restriction map $\mathrm{Res}_{L/K}: \Sha^2(K,\mathbb{Z}) \rightarrow \Sha^2(L,\mathbb{Z})$ is surjective and its kernel is $(\ires(C)\wedge n)$-torsion.
\end{lemma}

\begin{proof}
By restriction-corestriction, $\ker(\mathrm{Res}_{L/K})$ is killed by $n$. Moreover, since $\Sha^2(K,\Z)=\Sha^1(K,\Q/\Z)$, an element in $\ker(\mathrm{Res}_{L/K})$ corresponds to a subextension $K\subset L'\subset L$ that is locally trivial at every closed point of the curve $C$. Since $L=lK$, we can find an extension $k\subset l' \subset l$ such that $L'=l'K$. By the local triviality of $L'/K$, the field $l'$ has to be contained in the residue field of $k(v)$ for every $v \in C^{(1)}$. In particular, $[l':k]$ and $[L':K]$ divide $\ires(C)$. This shows that  $\ker(\mathrm{Res}_{L/K})$ is killed by $\ires(C)$, and hence by $\ires(C)\wedge n$.

In order to prove the surjectivity statement, consider an integral, regular, projective model $\mathcal{C}$ of $C$ such that its reduced special fibre $C_0$ is an SNC divisor. Let $c$ be the genus of the reduction graph of $\mathcal{C}$. According to Corollary 2.9 of \cite{Kato}, for each $m \geq 1$, we have an isomorphism:
$$\Sha^3(K,\mathbb{Z}/m\mathbb{Z}(2)) \cong (\mathbb{Z}/m\mathbb{Z})^{c}.$$
Hence, by Poitou-Tate duality, we also have:
$$\Sha^1(K,\mathbb{Z}/m\mathbb{Z}) \cong (\mathbb{Z}/m\mathbb{Z})^{c},$$
so that:
$$\Sha^2(K,\mathbb{Z}) \cong (\mathbb{Q}/\mathbb{Z})^{c}.$$
Since $l/k$ is unramified, the scheme $\mathcal{C} \times_{\mathcal{O}_k} \mathcal{O}_l$ is a suitable regular model of $C \times_k l$ and hence $\Sha^2(L,\mathbb{Z})$ is also isomorphic to $(\mathbb{Q}/\mathbb{Z})^c$. The surjectivity of $\mathrm{Res}_{L/K}$ then follows from the isomorphism $\Sha^2(K,\mathbb{Z}) \cong \Sha^2(L,\mathbb{Z})\cong (\mathbb{Q}/\mathbb{Z})^{c}$ and the finiteness of the exponent of $\ker(\mathrm{Res}_{L/K})$.
\end{proof}

\subsubsection{Step 5: Proof of Theorem \ref{corint} for smooth proper varieties}\label{subsubsec lisse}

In this step, we use Poitou-Tate duality to deduce Theorem \ref{corint} for smooth proper varieties from the previous steps:

\begin{theorem}\label{thint}
Let $l/k$ be a finite unramified extension and set $L:=lK$. Let $Z$ be a smooth proper integral $K$-variety. Then the quotient:
$$\K(K)/\langle N_{L/K}(\K(L)), N_2(Z/K) \rangle$$
is $\chi_K(Z,E)^2$-torsion for every coherent sheaf $E$ on $Z$.
\end{theorem}

\begin{proof}
Take $x\in \K(K)$. We want to prove that: $$\chi_K(Z,E)^2\cdot x\in \langle N_{L/K}(\K(L)), N_2(Z/K) \rangle.$$ 
First observe that, if $K'$ stands for the algebraic closure of $K$ in the function field of $Z$, then $Z$ has a structure of a smooth proper $K'$-variety and that $\chi_{K'}(Z,E)=[K':K]^{-1}\chi_K(Z,E)$. Therefore, by restriction-corestriction, we can assume that $K=K'$, and hence that $Z$ is geometrically integral. Moreover, by Proposition \ref{resid}, we may and do assume that $C$ has residual index $1$. \\

Let now $S$ be the (finite) set of places $v\in C^{(1)}$ such that $\partial_v x \neq 0$. Given a prime number $\ell$, since the curve $C$ has residual index $1$ and the field $k$ is large, we can find some point $w_\ell\in C^{(1)}\setminus S$ such that the residual degree $[k(w_\ell):k]_{\res}$ of $k(w_\ell)/k$ is prime to $\ell$. Moreover, by Proposition \ref{proploc}, we have 
\begin{equation}\label{inclusionatp}
    \chi_K(Z,E)\cdot \K(K_{w_\ell})\subseteq \left\langle N_{L_{w_\ell}/K_{w_\ell}}(\K(L_{w_\ell})), N_2(Z_{w_\ell}/K_{w_\ell}) \right\rangle.
\end{equation}
Before moving further, we need to prove the following lemma:

\begin{lemma}\label{lema feo}
Let $n=[l:k]$ with $l/k$ as in Theorem \ref{thint}. If $v_\ell(n)>v_\ell(\chi_K(Z,E))$, then there exists a finite extension $M^{(w_\ell)}$ of $K_{w_\ell}$ with residue field $m^{(w_\ell)}$ such that $Z(M^{(w_\ell)}) \neq \emptyset$ and $v_\ell([m^{(w_\ell)}:k(w_\ell)]_{\res})\leq v_\ell(\chi_K(Z,E))$.
\end{lemma}

\begin{proof}
By contradiction, assume that, for each finite extension $M$ of $K_{w_\ell}$ with residue field $m$ such that $Z(M)\neq\emptyset$, we have $v_\ell([m:k(w_\ell)]_{\res})>v_\ell(\chi_K(Z,E))$. By applying the residue map to \eqref{inclusionatp} and by denoting $l(w_\ell)$ the residue field of $L_{w_\ell}$, we get:
$$(K_{w_\ell}^\times)^{\chi_K(Z,E)}\subseteq \left\langle N_{l(w_\ell)/k(w_\ell)}(l(w_\ell)^\times);\, N_{m/k(w_\ell)}(m^\times) \mid v_\ell([m:k(w_\ell)]_{\res})>v_\ell(\chi_K(Z,E)) \right\rangle.$$
By applying the valuation $w_\ell$, we deduce that:
\begin{equation}\label{eqn lema feo}
    \chi_K(Z,E)\in \left\langle [l(w_\ell):k(w_\ell)]_{\res} ;\, [m:k(w_\ell)]_{\res} \mid v_\ell([m:k(w_\ell)]_{\res})>v_\ell(\chi_K(Z,E)) \right\rangle\subseteq\Z.
\end{equation}
Now, since $l/k$ is unramified, we have $[l:k]_{\res}=n$. Moreover, since $[k(w_\ell):k]_{\res} \wedge \ell =1 $, our hypothesis on $v_\ell(n)$ implies that:
\[v_\ell([l(w_\ell):k(w_\ell)]_{\res})\geq v_\ell([l:k]_{\res})>v_\ell(\chi_K(Z,E)).\]
Thus, every integer in:
$$\left\langle [l(w_\ell):k(w_\ell)]_{\res} ;\, [m:k(w_\ell)]_{\res} \mid v_\ell([m:k(w_\ell)]_{\res})>v_\ell(\chi_K(Z,E)) \right\rangle,$$
is divisible by $\ell^{v_\ell(\chi_K(Z,E))+1}$, which contradicts \eqref{eqn lema feo}.
\end{proof}

We keep the notation $n:=[l:k]$ and resume the proof of Theorem \ref{thint}. For $v\in C^{(1)}\setminus S$, we have 
\begin{equation}\label{apploc0}
    x\in N_{L_v/K_v}(\K(L_v))
\end{equation} 
by Lemma \ref{lemloc}. For $v\in S$, Proposition \ref{proploc} shows that we can find $M^{(v)}_1,\ldots,M^{(v)}_{r_v}$ finite extensions of $K_v$ such that $Z(M^{(v)}_i) \neq \emptyset$ for all $i$ and:
\begin{equation}\label{apploc}
\chi_K(Z,E)\cdot x\in \langle N_{L_v/K_v}(\K(L_v)); N_{M^{(v)}_i/K_v}(\K(M^{(v)}_i)), 1\leq i \leq r_v \rangle.
\end{equation}
By applying Proposition \ref{locglob} inductively, we can find, for each $v\in S$ and $1\leq i \leq r_v$, a finite extension $K_i^{(v)}$ of $K$ satisfying the following properties:
\begin{enumerate}
\item[(i)] $Z(K_i^{(v)})\neq \emptyset$;
    \item[(ii)]  there exists a $K$-embedding $K_i^{(v)}\hookrightarrow M_i^{(v)}$;
    \item[(iii)] there also exists a $K$-embedding $K_i^{(v)}\hookrightarrow M^{(w_\ell)}$, where $M^{(w_\ell)}$ is given by Lemma \ref{lema feo}, for each prime $\ell$ such that $v_\ell(n)>v_\ell(\chi_K(Z,E))$;
    \item[(iv)] for each pair $(v_0,i_0)$, the field $K_{i_0}^{(v_0)}$ is linearly disjoint to the composite field
    \[L_n\cdot \prod_{(v,i)\neq (v_0,i_0)} K_i^{(v)},\] over $K$, where $L_n$ stands for the composite of all cyclic extensions of $L$ that are locally trivial everywhere and whose degrees divide $n$. Note that $L_n$ is a finite extension of $L$ since $\Sha^1(L,\mathbb{Z}/n\mathbb{Z})$ is finite.
\end{enumerate}
Consider the Galois module $\hat{T}$ defined by the following exact sequence:
$$0 \rightarrow  \mathbb{Z}\rightarrow \mathbb{Z}[E/K] \rightarrow  \hat{T} \rightarrow 0,$$
where $E:=L\times \prod_{v,i} K_i^{(v)}$. To conclude, we introduce the composite field $K_{\mathcal{I}}=\prod_{v,i} K_i^{(v)}$ and we check the assumptions \ref{ld1}, \ref{h2}, \ref{h1}, \ref{h3}, \ref{h4} and \ref{h5} of Proposition \ref{shadiv} with $m=\chi_K(Z,E)$ and $$m'=\left |\ker(\mathrm{Res}_{LK_{\mathcal{I}}/L}:\Sha^2(L,\mathbb{Z}) \rightarrow \Sha^2(LK_{\mathcal{I}},\mathbb{Z}))\right|.$$
\begin{itemize}
\item[\ref{ld1}] The extension $L/K$ is obviously Galois. The fields $L$ and $K_{\mathcal{I}}$ are linearly disjoint over $K$ by (iv).

\item[\ref{h2}] This immediately follows from (iv).

\item[\ref{h1}] By proceeding exactly in the same way as in Lemma 4 of \cite{DW}, since we already have \ref{ld1}, one gets the injectivity of the restriction map:
$$H^2(K,\hat{T}) \rightarrow H^2(L,\hat{T})\oplus H^2(K_{\mathcal{I}},\hat{T}),$$
 and hence of:
$$\Sha^2(K,\hat{T}) \rightarrow \Sha^2(L,\hat{T})\oplus \Sha^2(K_{\mathcal{I}},\hat{T}).$$

\item[\ref{h3}] Let $C_{\mathcal{I}}$ be the smooth projective $k$-curve with fraction field $K_{\mathcal{I}}$. On the one hand, by (iii), given a prime $\ell$ such that $v_\ell(n)>v_\ell(\chi_K(Z,E))$, the field $K_{\mathcal{I}}$ can be seen as a subfield of $M^{(w_\ell)}$ and the inequality $v_\ell([m^{(w_\ell)}:k]_{\res})\leq v_\ell(\chi_K(Z,E))$ holds by Lemma \ref{lema feo}. We deduce that $v_\ell(\ires(C_{\mathcal{I}})) \leq v_\ell(\chi_K(Z,E))$ for such $\ell$. On the other hand, for any other prime number $\ell$, we have $v_\ell(n)\leq v_\ell(\chi_K(Z,E))$. We deduce that $\ires(C_{\mathcal{I}})\wedge n$ divides $m=\chi_K(Z,E)$, and hence \ref{h3} follows from Lemma \ref{res1}.

\item[\ref{h4}] This immediately follows from Lemma \ref{res1}.

\item[\ref{h5}] Since $L/K$ is Galois, \ref{h5} immediately follows from the choice of $m'$.
\end{itemize}
By Proposition \ref{shadiv}, we deduce that the group $\overline{\Sha^2(K,\hat{T})}$ is $\left((m\vee m')\wedge n\right)$-torsion. But by (iv), the fields $K_{\mathcal I}$ and $L_n$ are linearly disjoint over $K$, and hence, by the definition of $m'$, we have $m' \wedge n=1$, so that $(m\vee m')\wedge n=m \wedge n$. The group $\overline{\Sha^2(K,\hat{T})}$ is therefore $m$-torsion. If we set $\check{T}:=\mathrm{Hom}(\hat{T},\mathbb{Z})$ and $T:=\check{T}\otimes \mathbb{Z}(2)$, that is also the case of $\Sha^3(K,T)$ according to Poitou-Tate duality.

Now, by Lemma \ref{abstractlemma}, we may interpret $x$ as an element of $H^3(K,T)$. Equations \eqref{apploc0} and \eqref{apploc} together with assertion (ii) imply that $mx \in \Sha^3(K,T)$, which is $m$-torsion. Thus $m^2x=0 \in \Sha^3(K,T)$. This amounts to
\begin{align*}
    m^2x &\in \left\langle N_{L/K}(\K(L));\, N_{K_i^{(v)}/K}(\K(K_i^{(v)})),\, v\in S, \, 1\leq i \leq r_v \right\rangle\\
&\subseteq \left\langle N_{L/K}(\K(L)), N_2(Z/K) \right\rangle,
\end{align*}
the last inclusion being a consequence of (i).
\end{proof}

\subsubsection{Step 6: Proof of Theorem \ref{corint}}\label{subsubsec conclusion}

In this final step, we remove the smoothness assumption from the previous step and prove Theorem \ref{corint} for all proper varieties. For that purpose, we use the following variation of the dévissage technique given by Proposition 2.1 of \cite{Wit}:

\begin{proposition}[\cite{Wit}]\label{dev}
Let $K$ be a field and $r$ a positive integer. Let (P) be a property of proper $K$-varieties. Suppose we are given, for each proper $K$-variety $X$, an integer $m_X$. Make the following assumptions:
\begin{itemize}
    \item[(1)] For every morphism of proper $K$-schemes $Y \rightarrow X$, the integer $m_X$ divides $m_Y$.
    \item[(2)] For every proper $K$-scheme $X$ satisfying (P), the integer $m_X$ divides $\chi_K(X,\mathcal{O}_X)^r$.
     \item[(3)] For every proper and integral $K$-scheme $X$, there exists a proper $K$-scheme $Y$ satisfying (P) and a $K$-morphism $f:Y \rightarrow X$ with generic fiber $Y_{\eta}$ such that $m_X$ and $\chi_{K(X)}(Y_{\eta},\mathcal{O}_{Y_{\eta}})$ are coprime.
\end{itemize}
Then for every proper $K$-scheme $X$ and every coherent sheaf $E$ on $X$, the integer $m_X$ divides $\chi_K(X,E)^r$.
\end{proposition}

\begin{proof}
One can prove this result by following almost word by word the proof of Proposition 2.1 of \cite{Wit}. Alternatively, for each proper $K$-scheme $X$, write the prime decomposition of $m_X$:
$$m_X=\prod_p p^{\alpha_p},$$
and consider the integer
$$n_X:=\prod_p p^{\lceil \frac{\alpha_p}{r}\rceil}.$$
One can then easily check that the correspondence $X \mapsto n_X$ satisfies assumptions (1), (2) and (3) of Proposition 2.1 of \cite{Wit}. We deduce that $n_X|\chi_K(X,E)$, and hence that $m_X|\chi_K(X,E)^r$, for every proper $K$-scheme $X$ and every coherent sheaf $E$ on $X$. 
\end{proof}

\begin{proof}[Proof of Theorem \ref{corint}]
Given a proper $K
$-variety $Z$, we denote by $m_Z$ the exponent of the quotient $$\K(K)/\langle N_{L/K}(\K(L)), N_2(Z/K) \rangle.$$ We say that $Z$ has property (P) if it is smooth and integral. We have to check the three conditions (1), (2) and (3) of Proposition \ref{dev}. Condition (1) is straightforward. Condition (2) follows from Theorem \ref{thint}. Condition (3) follows from Hironaka's Theorem on resolution of singularities (Section 3.3 of \cite{Kollar}).
\end{proof}

\subsection{Proof of Main Theorem \ref{MTA}}

We can now deduce Main Theorem \ref{MTA} from Theorem \ref{corint}.

\begin{proof}[Proof of Main Theorem \ref{MTA}]
Fix two integers $n, d \geq 1$ such that $d^2 \leq n$ and a hypersurface $Z$ in $\mathbb{P}^n_{K}$ of degree $d$. By Lang's and Tsen's Theorems (Theorem 2a of \cite{Nagata} and Theorem 12 of \cite{Lang}), the field $k^{\mathrm{nr}}(C)$ is $C_2$. Since $d^2\leq n$, we deduce that there exists a finite unramified extension $l$ of $k$ such that $Z(lK) \neq \emptyset$. By Theorem \ref{corint}, the quotient:
$$\K(K)/\langle N_{lK/K}(\K(lK)), N_2(Z/K) \rangle=\K(K)/  N_2(Z/K) $$
is $ \chi_K(Z,\mathcal{O}_Z)^{2}$-torsion. But since $d \leq n$, Theorem III.5.1 of \cite{Hart} implies that $\chi_K(\mathbb{P}^n_{K},\mathcal{O}_{\mathbb{P}^n_{K}}(-d))=0$, and hence the exact sequence
$$0\rightarrow \mathcal{O}_{\mathbb{P}^n_{K}}(-d)\rightarrow  \mathcal{O}_{\mathbb{P}^n_{K}} \rightarrow i_* \mathcal{O}_Z \rightarrow 0,$$
where $i: Z \rightarrow \mathbb{P}^n_{K}$ stands for the closed immersion, gives:
$$\chi_K(Z,\mathcal{O}_Z)=\chi_K(\mathbb{P}^n_{K},\mathcal{O}_{\mathbb{P}^n_{K}})-\chi_K(\mathbb{P}^n_{K},\mathcal{O}_{\mathbb{P}^n_{K}}(-d))= 1.$$ Hence $\K(K)= N_2(Z/K)$. 
\end{proof}

\section{On the $C_1^2$ property for $p$-adic function fields}\label{sec4}

The goal of this section is to prove Main Theorem \ref{MTB}. Contrary to Main Theorem \ref{MTA}, for which we needed to deal with unramified extensions of $k$, here we will have to deal with ramified extensions of $k$. For that purpose, the key statement is given by the following theorem:

\begin{theorem}\label{extell}
Assume that $C$ has a rational point, let $\ell$ be a prime number, and fix a finite Galois totally ramified extension $l/k$ of degree $\ell$. Let $\mathcal{E}^0_{l/k}$ be the set of all finite ramified subextensions of $l^{\mathrm{nr}}/k$ and let $\mathcal{E}_{l/k}$ be the set of finite extensions $K'$ of $K$ of the form $K'=k'K$ for some $k'\in \mathcal{E}^0_{l/k}$. Then:
$$\K(K) = \langle N_{K'/K}(\K(K')) \, |\, K'\in \mathcal{E}_{l/k} \rangle.$$
\end{theorem}

Note that, given any two extensions $k'$ and $k''$ in $\mathcal{E}^0_{l/k}$ with $k'\subset k''$, the extension $k''/k'$ is unramified. This observation will be often used in the sequel.

\begin{remark}
We think that the assumption that $C$ has a rational point in Theorem \ref{extell} cannot be removed. To check that, we invite the reader to assume that $\iram(C)=\ell$. Then, given an integer $n \geq 1$, consider the set $\mathcal{E}_n^0$ whose elements are extensions of $k$ in $\mathcal{E}_{l/k}^0$ that are contained in the composite $l_n:=lk_n$, where $k_n$ is the degree $\ell^n$ unramified extension of $k$. Define the set $\mathcal{E}_n$ of finite extensions $K'$ of $K$ contained in $L_n:=l_nK$ that are of the form $K'=k'K$ for some $k'\in \mathcal{E}_n^0$ and consider the Galois module $\hat{T}_n$ defined by the exact sequence:
$$0 \rightarrow \mathbb{Z} \rightarrow \bigoplus_{K' \in \mathcal{E}_n} \mathbb{Z}[K'/K] \rightarrow \hat{T}_n \rightarrow 0.$$
By following the proof of Proposition \ref{shadiv2}, one can check that, if $K_1$ and $K_2$ are two distinct degree $\ell$ extensions of $K$ in $\mathcal{E}_n$, then the Tate-Shafarevich group $\Sha^2(K,\hat{T}_n)$ is the direct sum of the kernel of the map:
$$(\mathrm{Res}_{K_1/K},\mathrm{Res}_{K_2/K}):\Sha^2(K,\hat{T}_n) \rightarrow \Sha^2(K_1,\hat{T}_n)\oplus \Sha^2(K_2,\hat{T}_n)$$
 and of a divisible group, given by the kernel of the map:
$$\mathrm{Res}_{K_1K_2/K_1}-\mathrm{Res}_{K_1K_2/K_2}: \Sha^2(K_1,\hat{T}_n)\oplus \Sha^2(K_2,\hat{T}_n)\rightarrow \Sha^2(K_1K_2,\hat{T}_n).$$
In particular:
$$\overline{\Sha^2(K,\hat{T}_n)} \cong \ker \left(\Sha^2(K,\hat{T}_n) \rightarrow \Sha^2(K_1,\hat{T}_n)\oplus \Sha^2(K_2,\hat{T}_n)\right).$$
The computation of this kernel is a relatively simple (but a bit technical) exercise in the cohomology of finite groups, since it is contained in the group:
$$\ker \left(H^2(K,\hat{T}_n) \rightarrow H^2(L_n,\hat{T}_n)\right) \cong H^2(\mathrm{Gal}(L_n/K),\hat{T}) \cong H^2(\mathbb{Z}/\ell\mathbb{Z}\times \mathbb{Z}/\ell^n\mathbb{Z},\hat{T}).$$
In that way, one checks that $\overline{\Sha^2(K,\hat{T}_n)}$ is an $\mathbb{F}_\ell$-vector space of dimension at least $n\ell-n-1$. Moreover, the computation being very explicit, one can even check that the morphism $\overline{\Sha^2(K,\hat{T}_{n+1})}\rightarrow \overline{\Sha^2(K,\hat{T}_n)} $ induced by the natural projection $\hat{T}_{n+1}\rightarrow\hat{T}_{n}$ is always surjective. But then, by dualizing thanks to Poitou-Tate duality, this shows that the groups:
\begin{multline*}
    Q_n:=\ker ( \K(K)/\langle N_{K'/K}(\K(K')) \, |\, K'\in \mathcal{E}_{n} \rangle   \\ \rightarrow \prod_{v \in C^{(1)}} \K(K_v)/\langle N_{K'\otimes K_v/K_v}(\K(K'\otimes K_v)) \, |\, K'\in \mathcal{E}_{n} \rangle )
\end{multline*}
are all non-trivial and that the natural maps $Q_n \rightarrow Q_{n+1}$ are all injective. We deduce that the non-trivial elements of $Q_1$ provide non-trivial elements in the quotient:
$$\K(K)/\langle N_{K'/K}(\K(K')) \, |\, K'\in \bigcup_n \mathcal{E}_{n} \rangle = \K(K)/\langle N_{K'/K}(\K(K')) \, |\, K'\in \mathcal{E}_{l/k} \rangle.$$
\end{remark}

\subsection{Proof of Theorem \ref{extell}}

\subsubsection{Step 1: Solving the local problem}

The first step to prove Theorem \ref{extell} consists in settling an analogous statement over the completions of $K$. We start with the following lemma:

\begin{lemma}\label{ramunram}
Let $\ell$ be a prime number and let $l/k$ be a finite Galois totally ramified extension of degree $\ell$. Let $m/k$ be a totally ramified extension such that $ml/m$ is unramified. Then there exists $k' \in \mathcal{E}^0_{l/k}$ such that $k' \subset m$.
\end{lemma}

\begin{proof}
If $ml/m$ is trivial, then $m$ contains $l$ and we are done. Therefore we may and do assume that $ml/m$ has degree $\ell$. Denote by $k_\ell$ the unramified extension of $k$ with degree $\ell$ and set $l_\ell := l \cdot k_\ell$. The extension $l_\ell/k$ is Galois with Galois group $(\mathbb{Z}/\ell\mathbb{Z})^2$, and since $ml$ is unramified of degree $\ell$ over $m$, it contains both $k_\ell$ and $l_\ell$, so that $l_\ell$ is contained in $m'l$ for some finite subsextension $m'$ of $m/k$. But:
$$[m':k]\cdot [l_\ell:k]=\ell^2 [m':k] > \ell [m':k]=[m'l:k]=[m'l_\ell:k].$$ 
Hence the intersection $k':=m' \cap l_\ell$ is a degree $\ell$ totally ramified extension of $k$, and $k' \in \mathcal{E}^0_{l/k}$.
\end{proof}

\begin{proposition}\label{localextell}
Let $\ell$ be a prime number and let $l/k$ be a finite Galois totally ramified extension of degree $\ell$. Fix $v \in C^{(1)}$. Then:
$$\K(K_v) = \langle N_{K'\otimes_{K} K_v/K_v}(\K(K'\otimes_{K} K_v)) \, |\, K'\in \mathcal{E}_{l/k} \rangle.$$
\end{proposition}

\begin{proof}
Three different cases arise:
\begin{enumerate}
    \item the field $k(v)$ contains $l$;
        \item the extension $lk(v)/k(v)$ is unramified of degree $\ell$;
            \item the extension $lk(v)/k(v)$ is totally ramified of degree $\ell$.
\end{enumerate}

Case 1 is trivial, since:
$$\K(K_v) =  N_{lK \otimes_{K} K_v/K_v}(\K(lK \otimes_{K} K_v)).$$

Let us now consider case 2, and denote by $k(v)_{\mathrm{nr}}$ the maximal unramified subextension of $k(v)/k$. By Lemma \ref{ramunram}, since $lk(v)_{\mathrm{nr}}/k(v)_{\mathrm{nr}}$ is a Galois totally ramified extension of degree $\ell$ and $k(v)/k(v)_{\mathrm{nr}}$ is a totally ramified extension such that $k(v)l/k(v)$ is unramified, there exists a finite extension $m$ of $k(v)_{\mathrm{nr}}$ such that $m \in \mathcal{E}^0_{lk(v)_{\mathrm{nr}}/k(v)_{\mathrm{nr}}} \subset \mathcal{E}^0_{l/k}$ and $m \subset k(v)$. By setting $M:=mK$, we get that $M \in \mathcal{E}_{l/k}$ and that:
\begin{align*} 
\K(K_v)&=N_{M\otimes_K K_v/K_v}(\K(M\otimes_K K_v)) \\&\subset \langle N_{K'\otimes_{K} K_v/K_v}(\K(K'\otimes_{K} K_v)) \, |\, K'\in \mathcal{E}_{l/k} \rangle,
\end{align*}
as wished. \\

Let us finally consider case 3. To do so, fix a uniformizer $\pi$ of $k(v)$, and as before, let $k(v)_{\mathrm{nr}}$ be the maximal unramified subextension of $k(v)/k$. Denote by $k(v)_{\pi}^{\mathrm{ram}}$ the maximal abelian totally ramified extension of $k(v)$ associated to $\pi$ by Lubin-Tate theory. Since $l/k$ is abelian, the extension $lk(v)_{\pi}^{\mathrm{ram}}/k(v)_{\pi}^{\mathrm{ram}}$ must be unramified. Hence, by Lemma \ref{ramunram}, there exists a finite extension $m$ of $k(v)_{\mathrm{nr}}$ such that $m \in \mathcal{E}^0_{lk(v)_{\mathrm{nr}}/k(v)_{\mathrm{nr}}} \subset \mathcal{E}^0_{l/k}$ and $m \subset k(v)_{\pi}^{\mathrm{ram}}$. We deduce from Corollary 5.12 of \cite{Yoshida} that: 
$$    \pi \in N_{m \otimes_{k(v)_{\mathrm{nr}}} k(v)}((m \otimes_{k(v)_{\mathrm{nr}}} k(v))^\times)\subset \langle N_{k'\otimes_k k(v)/k(v)}((k'\otimes_k k(v))^\times) \, |\, k'\in \mathcal{E}^0_{l/k} \rangle.$$
This being true for every uniformizer $\pi$ of $k(v)$, we deduce that:
$$    k(v)^\times \subset \langle N_{k'\otimes_k k(v)/k(v)}((k'\otimes_k k(v))^\times) \, |\, k'\in \mathcal{E}^0_{l/k} \rangle,$$
and hence, by Lemma \ref{lemloc}:
$$
\K(K_v)= \langle N_{K'\otimes_{K} K_v/K_v}(\K(K'\otimes_{K} K_v)) \, |\, K'\in \mathcal{E}_{l/k} \rangle.$$
\end{proof}

\subsubsection{Step 2: Computation of a Tate-Shafarevich group}

The second step, which is slightly technical, consists in computing the Tate-Shafarevich groups of some finitely generated free Galois modules over $K$ associated to the fields in $\mathcal{E}_l$. Poitou-Tate duality will then allow us to obtain a local-global principle that will let us deduce Theorem \ref{extell} from Proposition \ref{localextell}.

\begin{proposition}\label{shadiv2}
Assume that $C$ has a rational point, and let $\ell$ be a prime number. Fix a finite Galois totally ramified extension $l/k$ of degree $\ell$. Given $K_1,\ldots,K_r$ in $\mathcal{E}_{l/k}$ so that the fields $K_1$ and $K_2$ are linearly disjoint over $K$, consider the Galois module $\hat{T}$ defined by the following exact sequence:
\begin{equation}\label{hat2}
0 \rightarrow  \mathbb{Z}\rightarrow \mathbb{Z}[E/K] \rightarrow  \hat{T} \rightarrow 0,
  \end{equation}
where $E:= K_1 \times \cdots \times K_r$. Then $\Sha^2(K,\hat{T})$ is divisible.
\end{proposition}

\begin{proof}
Consider the following complex:
\begin{equation}\label{keyseq2}
\xymatrix @R=.3pc{
\Sha^2(K,\hat{T}) \ar[r]^-{f_0}& \Sha^2(K_1,\hat{T})\oplus \Sha^2(K_{2},\hat{T}) \ar[r]^-{g_0}& \Sha^2(K_1K_2,\hat{T})\\
x \ar@{|->}[r]& (\mathrm{Res}_{K_1/K}(x),\mathrm{Res}_{K_{2}/K}(x))&\\
&(x,y) \ar@{|->}[r] & \mathrm{Res}_{K_1K_2/K_1}(x) - \mathrm{Res}_{K_1K_2/K_2}(y).
}
\end{equation}

We start by proving the following lemma:

\begin{lemma}
The morphism $f_0$ is injective.
\end{lemma}

\begin{proof}
Let $K_{\mathcal{I}}$ be the Galois closure of the composite of all the $K_i$'s. By inflation-restriction, there is an exact sequence:
$$0 \rightarrow H^2(K_{\mathcal{I}}/K,\hat{T}) \rightarrow H^2(K,\hat{T}) \rightarrow H^2(K_{\mathcal{I}},\hat{T}).$$
Take $v\in C(k)$ a rational point. Since the extension $K_{\mathcal{I}}/K$ is obtained by base change from an extension $k_{\mathcal{I}}$ of $k$, we have the equalities $\mathrm{Gal}(K_{\mathcal{I}}/K)=\mathrm{Gal}(k_{\mathcal{I}}/k)=\mathrm{Gal}(K_{\mathcal{I},v}/K_v)$. The previous inflation-restriction exact sequence therefore induces a commutative diagram with exact lines:
\[\xymatrix{
    0 \ar[r] & H^2(K_{\mathcal{I}}/K,\hat{T}) \ar[d]^{\cong}\ar[r] & H^2(K,\hat{T}) \ar[d]\ar[r] & H^2(K_{\mathcal{I}},\hat{T})\ar[d]\\
    0 \ar[r] & H^2(K_{\mathcal{I},v}/K_v,\hat{T}) \ar[r] & H^2(K_v,\hat{T}) \ar[r] & H^2(K_{\mathcal{I},v},\hat{T})
    }\]
in which the first vertical map is an isomorphism. We deduce that the restriction map:
$$\ker \left(  H^2(K,\hat{T})\rightarrow H^2(K_v,\hat{T})\right) \rightarrow \ker \left(  H^2(K_{\mathcal{I}},\hat{T})\rightarrow H^2(K_{\mathcal{I},v},\hat{T})\right)$$
is injective. Hence so is the restriction map:
$$\mathrm{Res}_{K_{\mathcal{I}}/K}:\Sha^2(K,\hat{T}) \rightarrow \Sha^2(K_{\mathcal{I}},\hat{T})$$
as well as the restriction maps:
\begin{gather*}
\mathrm{Res}_{K_1/K}:\Sha^2(K,\hat{T}) \rightarrow \Sha^2(K_1,\hat{T}),\\
\mathrm{Res}_{K_2/K}\Sha^2(K,\hat{T}) \rightarrow \Sha^2(K_2,\hat{T}),
\end{gather*}
since the former factors through these.
\end{proof}

Now observe that the complex \eqref{keyseq2} fits in the following commutative diagram, in which the first and second rows are obtained in the same way as the third:
\begin{equation}\label{diagram3}
\xymatrix{
& 0\ar[d] & 0\ar[d]\\
 \Sha^2(K,\mathbb{Z}) \ar[r]\ar[d]& \Sha^2(K_1,\mathbb{Z})\oplus \Sha^2(K_2,\mathbb{Z}) \ar@{->>}[r]\ar[d]& \Sha^2(K_1K_2,\mathbb{Z})\ar[d]\\
 \Sha^2(K, \mathbb{Z}[E/K]) \ar[r]^<<<<<{f} \ar[d]& \Sha^2(K_1, \mathbb{Z}[E/K])\oplus \Sha^2(K_2, \mathbb{Z}[E/K]) \ar[r]^<<<<<{g}\ar[d]& \Sha^2(K_1K_2, \mathbb{Z}[E/K])\ar[d]\\
 \Sha^2(K,\hat{T}) \ar@{^{(}->}[r]^>>>>>>>>>>>>>>>>{f_0}& \Sha^2(K_1,\hat{T})\oplus \Sha^2(K_2,\hat{T}) \ar[r]^>>>>>>>>>>>{g_0}\ar[d]& \Sha^2(K_1K_2,\hat{T})\ar[d]\\
 & 0 & 0.
}
\end{equation}
The second and third columns are exact since the exact sequence (\ref{hat2}) splits over $K_1$, $K_2$ and $K_1K_2$.  The lines are all complexes. In the first one, the second arrow is surjective since the restriction map:
$$ \Sha^2(K_1,\mathbb{Z}) \rightarrow \Sha^2(K_1K_2,\mathbb{Z})$$
is an isomorphism by Lemma \ref{res1} and $C$ has a rational point. As for the second line, we have the following lemma:

\begin{lemma}
The second line of diagram \eqref{diagram3} is exact.
\end{lemma}

\begin{proof}
For $1 \leq \alpha \leq r$, write:
\begin{align*}    
K_1 \otimes_K K_\alpha &= \prod_\beta L_{\alpha\beta}\\
    K_2 \otimes_K K_\alpha &= \prod_\gamma M_{\alpha\gamma}\\
    L_{\alpha\beta}\otimes_{K_\alpha} M_{\alpha\gamma} &= \prod_{\delta} N_{\alpha\beta\gamma\delta}
\end{align*} 
for some fields $L_{\alpha\beta}$, $M_{\alpha\gamma}$ and $N_{\alpha\beta\gamma\delta}$. By Shapiro's Lemma, the second line of \eqref{diagram3} can be identified with the following complex:
\[\xymatrix{
 \bigoplus_\alpha \Sha^2(K_\alpha,\mathbb{Z}) \ar[d]^f \\
\bigoplus_{\alpha,\beta} \Sha^2(L_{\alpha\beta},\mathbb{Z}) \oplus \bigoplus_{\alpha,\gamma} \Sha^2(M_{\alpha\gamma},\mathbb{Z}) \ar[d]^g \\
\bigoplus_{\alpha,\beta,\gamma,\delta} \Sha^2(N_{\alpha\beta\gamma\delta},\mathbb{Z}) 
}\]
where $f$ is given by:
$$(x_\alpha) \mapsto \left( \left(\mathrm{Res}_{L_{\alpha\beta}/K_\alpha}(x_\alpha)\right)_{\alpha\beta},  \left(\mathrm{Res}_{M_{\alpha\gamma}/K_\alpha}(x_\alpha)\right)_{\alpha\gamma}\right),$$
and $g$:
$$\left((y_{\alpha\beta})_{\alpha,\beta},(z_{\alpha\gamma})_{\alpha,\gamma})\right) \mapsto \left(\mathrm{Res}_{N_{\alpha\beta\gamma\delta}/L_{\alpha\beta}}(y_{\alpha\beta})-\mathrm{Res}_{N_{\alpha\beta\gamma\delta}/M_{\alpha\gamma}}(z_{\alpha\gamma})\right)_{\alpha\beta\gamma\delta},$$
Fix $\left((y_{\alpha\beta})_{\alpha,\beta},(z_{\alpha\gamma})_{\alpha,\gamma})\right) \in \ker (g)$. Then:
$$\mathrm{Res}_{N_{\alpha\beta\gamma\delta}/L_{\alpha\beta}}(y_{\alpha\beta})=\mathrm{Res}_{N_{\alpha\beta\gamma\delta}/M_{\alpha\gamma}}(z_{\alpha\gamma})$$
for all $\alpha,\beta,\gamma,\delta$. But the restrictions $\mathrm{Res}_{L_{\alpha\beta}/K_{\alpha}}$, $\mathrm{Res}_{M_{\alpha\gamma}/K_{\alpha}}$, $\mathrm{Res}_{N_{\alpha\beta\gamma\delta}/L_{\alpha\beta}}$ and $\mathrm{Res}_{N_{\alpha\beta\gamma\delta}/M_{\alpha\gamma}}$ are all isomorphisms by Lemma \ref{res1} and they fit into a commutative diagram:
\[\xymatrix{
\Sha^2(K_\alpha,\mathbb{Z})\ar[rr]^-{\mathrm{Res}_{L_{\alpha\beta}/K_{\alpha}}}\ar[d]_-{\mathrm{Res}_{M_{\alpha\gamma}/K_{\alpha}}} && \Sha^2(L_{\alpha\beta},\mathbb{Z})\ar[d]^-{\mathrm{Res}_{N_{\alpha\beta\gamma\delta}/L_{\alpha\beta}}}\\
\Sha^2(M_{\alpha\gamma},\mathbb{Z})\ar[rr]_-{\mathrm{Res}_{N_{\alpha\beta\gamma\delta}/M_{\alpha\gamma}}} && \Sha^2(N_{\alpha\beta\gamma\delta},\mathbb{Z}).
}\]
We deduce that, for each $\alpha$, there exists $x_{\alpha}\in \Sha^2(K_\alpha,\mathbb{Z})$ such that:
\begin{gather*}
    \forall \beta,\;\; \mathrm{Res}_{L_{\alpha\beta}/K_{\alpha}}(x_{\alpha})=y_{\alpha\beta},\\
    \forall \gamma,\;\; \mathrm{Res}_{M_{\alpha\gamma}/K_{\alpha}}(x_{\alpha})=z_{\alpha\gamma}.
\end{gather*}
In other words, $\left((y_{\alpha\beta})_{\alpha,\beta},(z_{\alpha\gamma})_{\alpha,\gamma})\right) \in \im (f)$.
\end{proof}

With all the gathered information, a simple diagram chase in \eqref{diagram3} shows that the morphism $ \Sha^2(K,\mathbb{Z}[E/K]) \rightarrow \Sha^2(K,\hat{T}) $ is surjective. But as recalled at the end of section \ref{sec prel}, the group $\Sha^2(K, \mathbb{Z}[E/K])$ is divisible. Hence so is $\Sha^2(K,\hat{T}) $.
\end{proof}

\subsubsection{Step 3: Proof of Theorem \ref{extell}}

We can finally prove Theorem \ref{extell} by using Poitou-Tate duality.

\begin{proof}[Proof of Th. \ref{extell}]
Take $x\in \K(K)$. By Proposition \ref{localextell}, we have:
$$\K(K_v) = \langle N_{K'\otimes_{K} K_v/K_v}(\K(K'\otimes_{K} K_v)) \, |\, K'\in \mathcal{E}_{l/k} \rangle$$
for all $v \in C^{(1)}$. Hence we can find $K_1,\ldots,K_r \in \mathcal{E}_{l/k}$ such that:
\begin{multline}\label{xappart}
x \in \ker ( \K(K)/\langle N_{K_i/K}(\K(K_i)) \; | \; 1\leq i \leq r \rangle  \\ \rightarrow \prod_{v\in C^{(1)}} \K(K_v)/\langle N_{K_i\otimes_K K_v/K_v}(\K(K_i\otimes_K K_v)) \; | \; 1\leq i \leq r \rangle).
\end{multline}
Moreover, up to enlarging the family $(K_i)_i$, we may and do assume that $K_1$ and $K_2$ are linearly disjoint. Consider the étale $K$-algebra $E:= K_1 \times \cdots \times K_r$ and the Galois module $\hat{T}$ defined by the following exact sequence:
\begin{equation*}
0 \rightarrow  \mathbb{Z}\rightarrow \mathbb{Z}[E/K] \rightarrow  \hat{T} \rightarrow 0.
  \end{equation*}
 Set $\check{T}:=\mathrm{Hom}(\hat{T},\mathbb{Z})$ and $T:= \check{T} \otimes \mathbb{Z}(2)$. By Lemma \ref{abstractlemma}, equation \eqref{xappart} can be rewritten as:
 $$x \in \Sha^3(K, T).$$
 But, by Poitou-Tate duality, $\Sha^3(K,T)$ is dual to $\overline{\Sha^2(K,\hat{T})}$, and by Proposition \ref{shadiv2}, the group $\Sha^2(K,\hat{T})$ is divisible. We deduce that $\Sha^3(K,T)=0$, and hence that: $$x \in \langle N_{K_i/K}(\K(K_i)) \; | \; 1\leq i \leq r \rangle \subset \langle N_{K'/K}(\K(K')) \, |\, K'\in \mathcal{E}_{l/k} \rangle.$$
\end{proof}

\subsection{Proof of Main Theorem \ref{MTB}}

By combining Theorems \ref{corint} and \ref{extell}, we can now settle the following theorem, from which we will deduce Main Theorem \ref{MTB}:

\begin{theorem}\label{endth}
Let $K$ be the function field of a smooth projective curve $C$ defined over a $p$-adic field $k$. Let $l/k$ be a finite Galois extension and set $L:=lK$. Let $Z$ be a proper $K$-variety. If $s_{l/k}$ stands for the number of (not necessarily distinct) prime factors of the ramification degree of $l/k$, then the quotient:
$$\K(K)/\langle N_{L/K}(\K(L)), N_2(Z/K) \rangle$$
is $\iram(C)\cdot \chi_K(Z,E)^{2s_{l/k}+4}$-torsion for every coherent sheaf $E$ on $Z$.
\end{theorem}

\begin{proof}
We first assume that $C$ has a rational point, and we prove that
$$\K(K)/\langle N_{L/K}(\K(L)), N_2(Z/K) \rangle$$
is $\chi_K(Z,E)^{2s_{l/k}+2}$-torsion for every coherent sheaf $E$ on $Z$ by induction on $s_{l/k}$. The case $s_{l/k}=0$ immediately follows from Theorem \ref{corint}. We henceforth assume now that $s_{l/k}>0$. Let $l_{\mathrm{nr}}$ be the maximal unramified subextension of $l/k$ and set $L_{\mathrm{nr}}:=l_{\mathrm{nr}}K$. Theorem \ref{corint} ensures then that the quotient:
$$\K(K)/\langle N_{L_{\mathrm{nr}}/K}(\K(L_{\mathrm{nr}})), N_2(Z/K) \rangle$$
is $\chi_K(Z,E)^2$-torsion. Now, the extension $l/l_{\mathrm{nr}}$ is Galois and totally ramified. Since finite extensions of local fields are solvable, we can find a Galois totally ramified extension $m/l_{\mathrm{nr}}$ contained in $l$ and of prime degree $\ell$. Set $M:=mK$. By Theorem \ref{extell}, we have:
 $$\K(L_{\mathrm{nr}})=\langle N_{K'/L_{\mathrm{nr}}}(\K(K')) \, |\, K'\in \mathcal{E}_{m/l_{\mathrm{nr}}} \rangle.$$
 But for each $k' \in \mathcal{E}^0_{m/l_{\mathrm{nr}}}$, the ramification degree of $lk'/k'$ strictly divides that of $l/k$. Hence, by induction, the group:
 $$\K(K')/\langle N_{LK'/K'}(\K(LK')), N_2(Z/K') \rangle$$
is $\chi_K(Z,E)^{2s_{l/k}}$-torsion for each $K'\in \mathcal{E}_{m/l_{\mathrm{nr}}}$. We deduce that:
$$\K(K)/\langle N_{L/K}(\K(L)), N_2(Z/K) \rangle$$
is $\chi_K(Z,E)^{2s_{l/k}+2}$-torsion, which finishes the induction.\\

We do not assume anymore that $C$ has a rational point. Let $k_1,\ldots,k_r$ be finite extensions of $k$ over which $C$ acquires rational points and such that the g.c.d.'s of their ramification degrees is $\iram(C)$. For each $i$, let $k_{i,\mathrm{nr}}$ be the maximal unramified extension of $k$ contained in $k_i$, and set $K_i:=k_iK$ and $K_{i,\mathrm{nr}}:=k_{i,\mathrm{nr}}K$. Theorem \ref{corint} ensures that the quotient:
$$\K(K)/\langle N_{K_{i,\mathrm{nr}}/K}(\K(K_{i,\mathrm{nr}})), N_2(Z/K) \rangle$$
is $\chi_K(Z,E)^2$-torsion. Moreover, a restriction-corestriction argument shows that the quotient:
$$\K(K_{i,\mathrm{nr}})/N_{K_{i}/K_{i,\mathrm{nr}}}(\K(K_{i})) $$
is $[k_i:k_{i,\mathrm{nr}}]$-torsion. Since $[k_i:k_{i,\mathrm{nr}}]$ is the ramification degree of $k_i/k$, we deduce that:
$$\K(K)/\langle N_{K_{1}/K}(\K(K_{1})), \ldots , N_{K_{r}/K}(\K(K_{r})), N_2(Z/K) \rangle$$
is $\iram(C)\cdot \chi_K(Z,E)^2$-torsion. But $C$ has rational points over all the $k_i$'s. Hence, by the first case, the quotients:
$$\K(K_i)/\langle N_{LK_i/K_i}(\K(LK_i)), N_2(Z/K_i) \rangle$$
are all $\chi_K(Z,E)^{2s_{l/k}+2}$-torsion. We deduce that:
$$\K(K)/\langle N_{L/K}(\K(L)), N_2(Z/K) \rangle$$
is $\iram(C)\cdot \chi_K(Z,E)^{2s_{l/k}+4}$-torsion.
\end{proof}

Applying this to the context of the $C_1^2$-property, we get the following result:

\begin{corollary} \label{MTBbis}
Let $K$ be the function field of a smooth projective curve $C$ defined over a $p$-adic field $k$. Then, for each $n, d \geq 1$ and for each hypersurface $Z$ in $\mathbb{P}^n_{K}$ of degree $d$ with $d \leq n$, the quotient $\K(K)/N_2(Z/K)$ is killed by $\iram(C)$.
\end{corollary}

\begin{proof}
Let $Z$ be a hypersurface in $\mathbb{P}^n_{K}$ of degree $d$ with $d \leq n$. By Tsen's Theorem, the field $\overline{k}(C)$ is $C_1$. Since $d \leq n$, we deduce that there exists a finite extension $l$ of $k$ such that $Z(lK) \neq \emptyset$. By Theorem \ref{endth}, the quotient:
$$\K(K)/\langle N_{lK/K}(\K(lK)), N_2(Z/K) \rangle=\K(K)/N_2(Z/K) $$
is $\iram(C)\cdot \chi_K(Z,\mathcal{O}_Z)^{2s_{l/k}+4}$-torsion. But since $d \leq n$, Theorem III.5.1 of \cite{Hart} and the exact sequence
$$0\rightarrow \mathcal{O}_{\mathbb{P}^n_{K}}(-d)\rightarrow  \mathcal{O}_{\mathbb{P}^n_{K}} \rightarrow i_* \mathcal{O}_Z \rightarrow 0,$$
where $i: Z \rightarrow \mathbb{P}^n_{K}$ stands for the closed immersion, imply that:
$$\chi_K(Z,\mathcal{O}_Z)=\chi_K(\mathbb{P}^n_{K},\mathcal{O}_{\mathbb{P}^n_{K}})-\chi_K(\mathbb{P}^n_{K},\mathcal{O}_{\mathbb{P}^n_{K}}(-d))=  1.$$ Hence the quotient $\K(K)/  N_2(Z/K) $ is $\iram(C)$-torsion.
\end{proof}

Main Theorem \ref{MTB} can now be immediately deduced from the following corollary:

\begin{corollary} \label{MTBbisbis}
Let $K$ be the function field of a smooth projective curve $C$ defined over a $p$-adic field $k$. Assume that $\iram(C)=1$. Then, for each $n, d \geq 1$ and for each hypersurface $Z$ in $\mathbb{P}^n_{K}$ of degree $d$ with $d \leq n$, we have $N_2(Z/K)=\K(K)$.
\end{corollary}

\begin{remark}
By section 9.1 of \cite{BLR}, the assumption that $\iram(C)=1$ automatically holds when $C$ has an irreducible proper flat regular model whose special fibre has multiplicity 1.
\end{remark}

\end{document}